\documentclass[noinfoline]{imsart}

\usepackage[a4paper,top=3cm, bottom=3cm, left=2.5cm, right=2.5cm]{geometry}
\usepackage{fancyhdr}
\RequirePackage{amsthm,amsmath,amsfonts,amssymb}

\allowdisplaybreaks
\RequirePackage[colorlinks,citecolor=blue,urlcolor=blue]{hyperref}

\usepackage{mathrsfs} 

\usepackage{graphicx}
\graphicspath{{Images/}}
\usepackage{float}
\usepackage[figuresright]{rotating} %sidewaysfigure
\usepackage{subcaption}
\usepackage{pdfpages}
\usepackage{wrapfig}

\usepackage{tikz}
\usetikzlibrary{arrows.meta} 

\usepackage{array}
\usepackage{booktabs} 
\usepackage{enumerate}
\usepackage{enumitem} %for no indent
\usepackage{multicol}
\usepackage{algorithm}
\usepackage[noend]{algpseudocode}

\renewcommand{\b}[1]{\mathbb{#1}}
\newcommand{\m}[1]{\mathbf{#1}}
\renewcommand{\c}[1]{\mathcal{#1}}
\newcommand{\innpr}[2]{{\left\langle {#1},{#2}\right\rangle}}
\newcommand{\rd}{\mathrm{d}}

\newcommand{\indep}{\perp\mkern-9.5mu\perp}

\DeclareMathOperator*{\argmax}{\arg\max}
\DeclareMathOperator{\tr}{tr}
\DeclareMathOperator{\KL}{KL}

\newcommand{\vertiii}[1]{{\left\vert\kern-0.25ex\left\vert\kern-0.25ex\left\vert #1 
    \right\vert\kern-0.25ex\right\vert\kern-0.25ex\right\vert}} %Operator norm

\newcommand{\vertii}[1]{{\vert\kern-0.25ex\vert\kern-0.25ex\vert #1 
    \vert\kern-0.25ex\vert\kern-0.25ex\vert}} %Operator norm

\usepackage[capitalize]{cleveref}

\theoremstyle{plain}
\newtheorem{theorem}{Theorem}[section]
\newtheorem{lemma}[theorem]{Lemma}
\newtheorem{corollary}[theorem]{Corollary}
\newtheorem{proposition}[theorem]{Proposition}
\newtheorem{definition}{Definition}[section]
\newtheorem{example}{Example}[section]
\newtheorem{conjecture}{Conjecture}[section]
\newtheorem{remark}{Remark}[section]

\begin{document}

%%%%%%%%%%%%%%%%%%%
%%%For arXiv Preprint
\pagestyle{fancy}
\fancyhead{}
\renewcommand{\headrulewidth}{0pt} 
\fancyhead[CE]{Ho Yun}
\fancyhead[CO]{Spectral Shrinkage of Gaussian Entropic Optimal Transport}
%%%%%%%%%%%%%%%%%%%

\begin{frontmatter}
\title{{\large Spectral Shrinkage of Gaussian Entropic Optimal Transport}}

\begin{aug}
\author[A]{\fnms{Ho} \snm{Yun}\ead[label=e1]{ho.yun@epfl.ch}}

\thankstext{t1}{Research supported by a Swiss National Science Foundation grant awarded to Victor M. Panaretos.}
\address[A]{Ecole Polytechnique F\'ed\'erale de Lausanne, \printead{e1}}
\end{aug}

\begin{abstract}
We present a functional calculus treatment of Entropic Optimal Transport (EOT) between Gaussian measures on separable Hilbert spaces, providing a unified framework that handles infinite-dimensional degeneracy. By leveraging the notion of proper alignment and the Schur complement, we reveal that the Gaussian EOT solution operates as a precise \textit{spectral shrinkage}: the optimal coupling is uniquely determined by contracting the spectrum of the correlation operator via a universal scalar function. This geometric insight facilitates an algorithmic shift from iterative fixed-point schemes (e.g., Sinkhorn) to direct algebraic computation, enabling efficient multi-scale analysis, where a single spectral decomposition allows for the exact evaluation of entropic costs across arbitrary regularization parameters $\varepsilon > 0$ at negligible additional cost. Furthermore, we investigate the asymptotic behavior as $\varepsilon \downarrow 0$ in settings where the unregularized Optimal Transport problem admits non-unique solutions. We establish a selection principle that the regularized limit converges to the most diffusive optimal coupling --characterized as the centroid of the convex set of optimal Kantorovich plans. This demonstrates that in degenerate regimes, the entropic limit systematically rejects deterministic Monge solutions (extremal points) in favor of the optimal solution with minimal Hilbert-Schmidt correlation, effectively filtering out spurious correlations in the null space. Finally, we derive stability bounds and convergence rates, recovering established parametric rates ($\varepsilon \log(1/\varepsilon)$) in finite dimensions while identifying distinct non-parametric rates dependent on spectral decay in infinite-dimensional settings.
\end{abstract}


\end{frontmatter}

\maketitle

\section{Introduction}
The Entropic Optimal Transport (EOT) problem has emerged as a computational and theoretical cornerstone in modern data science, bridging the gap between optimal transport (OT) theory and efficient numerical algorithms via the Schr\"{o}dinger bridge framework \cite{leonard2012schrodinger, chen2021optimal, peyre2019computational}. Within this landscape, the Gaussian setting serves as a canonical prototype, offering one of the few regimes where the EOT problem admits closed-form algebraic solutions \cite{janati2020entropic}. However, the geometric structure of these solutions is often obscured by the dominant analytical approach: matrix differential calculus \cite{janati2020entropic, mallasto2022entropy, minh2023entropic, bunne2023schrodinger, freulon2025entropic}. By treating covariance matrices as generic variables to be differentiated, standard derivations typically yield complex Riccati-type algebraic equations that hide the intrinsic geometry of the problem.

We investigate the Gaussian EOT problem between two arbitrary centered Gaussian measures $\mu = N(0, \m{A})$ and $\nu = N(0, \m{B})$ on a separable Hilbert space $\c{H}$, where the covariances $\m{A}, \m{B}$ are non-negative definite (n.n.d.) and trace-class \cite{da2006introduction}. We assume centered measures without loss of generality, as the mean vectors contribute only a trivial affine shift separable from the covariance transport. In contrast to the literature of Gaussian EOT, which often imposes non-degeneracy assumptions (i.e., $\c{N}(\m{A})=\{0\}$) to ensure the differentiability of the objective function, our framework imposes no nullity conditions nor inclusion assumptions (e.g., $\c{N}(\m{A}) \subseteq \c{N}(\m{B}))$. Instead, we propose a coordinate-free geometric formulation based on properly aligned square roots. This perspective trivializes the algebraic complexity, replacing opaque matrix equations with explicit spectral operations.

\paragraph{Proper Alignment and Bures-Wasserstein Geometry} 
The core difficulty in comparing two covariance operators $\m{A}$ and $\m{B}$ lies in the ambiguity of their square roots, which we call the Green's operators $\m{G}, \m{M}$ satisfying $\m{G} \m{G}^{*} = \m{A}$ and $\m{M} \m{M}^{*} = \m{B}$. While $\m{A}^{1/2}$ is undoubtedly the canonical choice for a single operator, the geometry of the 2-Wasserstein distance suggests that $\m{A}$ and $\m{B}$ must be treated interactively. Recall that the 2-Wasserstein distance is given by \cite{gelbrich1990formula, bhatia2019bures, takatsu2011wasserstein, panaretos2020invitation}:
\begin{align*}
    \c{W}_{2}^{2}(\m{A}, \m{B}) = \tr [\m{A}] + \tr [\m{B}] - 2 \tr [(\m{A}^{1/2} \m{B} \m{A}^{1/2})^{1/2}] \le \vertii{\m{A}^{1/2} - \m{B}^{1/2}}_{2}^{2},
\end{align*}
where $\vertii{\cdot}_{2}$ denotes the Hilbert-Schmidt norm. The equality holds if and only if $\m{A}^{1/2} \m{B}^{1/2} \succeq \m{0}$, a condition that fails in general unless the operators commute. This discrepancy motivates our notion of \textit{proper alignment}: among the class of Green's operators, we select the pair satisfying $\m{G}^{*} \m{M} \succeq \m{0}$. This choice of pair recovers the 2-Wasserstein distance as a Hilbert-Schmidt distance $\c{W}_{2}(\m{A}, \m{B}) = \vertii{\m{G} - \m{M}}_{2}$, and renders the analysis robust to degeneracy—a regime where classical Kantorovich duality or density-based arguments often become intractable due to the failure of Brenier's theorem \cite{villani2008optimal, villani2021topics, yun2025gaussian}. From the perspective of shape analysis \cite{dryden2016statistical, pigoli2014distances, masarotto2019procrustes}, the Green's operator can be viewed as a preshape lying in the fibre of the covariance. The 2-Wasserstein distance then corresponds naturally to the Procrustes distance: it is the infimum of the Hilbert-Schmidt distance over the unitary orbits of $\m{G}$ and $\m{M}$. Nonetheless, while the operator product $\m{G} \m{M}^{*}$ may vary with the specific choice of the aligned pair, the spectrum of $\m{G}^{*} \m{M}$ remains invariant (coinciding with that of $(\m{A}^{1/2} \m{B} \m{A}^{1/2})^{1/2}$). Consequently, any optimal Kantorovich coupling admits a covariance of the form:
\begin{align*}
    \m{\Sigma} = \left(
    \begin{array}{c|c}
    \m{A} & \m{G} \m{M}^* \\
    \hline
    \m{M} \m{G}^* & \m{B}
    \end{array}
    \right).
\end{align*}
Uniqueness is guaranteed if and only if the generalized Schur complement (denoted $\m{B}/\m{A}$ or $\m{A}/\m{B}$ in our coordinate-free approach) vanishes, implying no residual information in one marginal remains unexplained by the other at least in one direction \cite{yun2025gaussian}.

\paragraph{Spectral Characterization of Entropic Solutions} Leveraging this alignment, we characterize the entire class of joint Gaussian couplings equivalent to the independent coupling $\mu \otimes \nu$ through the Felman-H\'{a}jek dichotomy. We show that minimizing the entropic cost is equivalent to maximizing a quadratic \emph{entropic interaction} term with the Fredholm determinant:
\begin{align*}
    \c{P}_{\varepsilon}(\m{R}) = 2 \tr [\m{R} (\m{G}^{*} \m{M})] + \varepsilon \log \det(\m{I} - \m{R}^* \m{R}), \quad \varepsilon > 0,
\end{align*}
where $\m{R}$ is the Green's correlation operator with respect to $(\m{G}, \m{M})$. Unlike the canonical non-interactive pair $(\m{A}^{1/2}, \m{B}^{1/2})$, the properly aligned pair decouples this optimization spectrally, admitting a closed-form solution $\m{R}_{\varepsilon} = f_{\varepsilon} (\m{G}^{*} \m{M})$. Here, $f_{\varepsilon}(x) = 2x/(\sqrt{4x^{2}+\varepsilon^{2}}+\varepsilon) \in [0, 1)$ is a universal \textit{spectral shrinkage function} that monotonically shrinks the correlation eigenvalues as regularization $\varepsilon$ increases. This result demystifies the Gaussian EOT solution: it is simply a spectral filter applied to the optimal Kantorovich coupling, yielding the unique solution invariant to the specific choice of properly aligned pair:
\begin{align*}
    \pi_{\varepsilon} = N(\m{0}, \m{\Sigma}_{\varepsilon}) \in \Pi(\mu, \nu), \quad \m{\Sigma}_{\varepsilon} = \left(
    \begin{array}{c|c}
    \m{A} & \m{G} \m{R}_{\varepsilon} \m{M}^* \\
    \hline
    \m{M} \m{R}_{\varepsilon} \m{G}^* & \m{B}
    \end{array}
    \right).
\end{align*}

Beyond theoretical clarity, this framework enables a fundamental algorithmic shift. While standard Gaussian EOT solvers rely on Sinkhorn iterations \cite{janati2020entropic, mallasto2022entropy}, which are iterative fixed-point schemes, our approach reduces the problem to a direct algebraic computation. By recognizing that the cost is determined purely by the singular spectrum of $\m{G}^{*} \m{M}$, the exact solution can be recovered using standard linear algebra routines (e.g., SVD) in finite time, eliminating iterative approximation errors. This advantage is particularly pronounced in multi-scale simulations: while the Sinkhorn-based algorithm reiteraterate for each $\varepsilon > 0$ where the required iterations typically diverge as $\varepsilon \downarrow 0$ \cite{cuturi2013sinkhorn, ghosal2025convergence, chizat2020faster}, in our approach once the eigendecomposition of $\m{G}^{*} \m{M}$ is computed for any properly aligned pair—a single $\c{O}(d^{3})$ operation when $\c{H} = \b{R}^{d}$—the solution for any subsequent $\varepsilon > 0$ can be evaluated in $\c{O}(d)$ time.

\paragraph{Limiting Behavior and Diffusivity} 
Our framework elucidates the convergence behavior as
$\varepsilon \to 0$. The sequence of spectral shrinkage functions $\{f_{\varepsilon}\}_{\varepsilon > 0}$ converges pointwise to the indicator function $f_{\varepsilon}(x) \to f_{0}(x) := I(x > 0)$, but fails to converge uniformly. This discontinuity at zero explains the emergence of non-unique optimal Kanotorovich solutions in degenerate regimes: the behavior on the null space of $\m{G}^{*} \m{M}$ becomes ill-posed. Crucially, we establish that the EOT coupling limit $\pi_{0} = \lim_{\varepsilon \to 0} \pi_{\varepsilon}$ with respect to the 2-Wasserstein distance is the \emph{most diffusive} Gaussian solution among all possible optimal Kantorovich couplings. Specifically, $\pi_{0}$ corresponds to the centroid of the convex set of optimal couplings, constructed via the orthogonal projection $\m{R}_{0} = f_{0}(\m{G}^{*} \m{M})$ onto the closure of the range of $\m{G}^{*} \m{M}$.
This implies that, whenever the Schur complement does not vanish, the optimal Kantorovich solution is non-unique, among which the limit of EOT coupling is never derived from a Monge solution (extremal point of the solution set) \cite{leonard2012schrodinger, yun2025gaussian}:
\begin{align*}
    \pi_{0} = N(\m{0}, \m{\Sigma}_{0}) \in \Pi(\mu, \nu), \quad \m{\Sigma}_{0} = \left(
    \begin{array}{c|c}
    \m{A} & \m{G} \m{R}_{0} \m{M}^* \\
    \hline
    \m{M} \m{R}_{0} \m{G}^* & \m{B}
    \end{array}
    \right).
\end{align*}

Finally, owing to the explicit form of $\m{R}_{\varepsilon} = f_{\varepsilon} (\m{G}^{*} \m{M})$, we provide sharp stability bounds for the entropic bias $\c{W}_{2, \varepsilon} (\m{A}, \m{B})^{2} - \c{W}_{2} (\m{A}, \m{B})^{2}$. In the Euclidean setting $\c{H} = \b{R}^{d}$, we recover the convergence rate of $\varepsilon \log(1/\varepsilon)$, a well-established rate in various setups under certain regularity conditions \cite{leonard2012schrodinger, carlier2017convergence, nutz2022entropic, carlier2023convergence, eckstein2022quantitative, ghosal2022stability, janati2020entropic, goldfeld2020gaussian}. However, in infinite-dimensional settings, we show that the rate depends on the spectral decay of $(\m{A}^{1/2} \m{B} \m{A}^{1/2})^{1/2}$, exhibiting slower rates. For instance, we show that the sharpest rate for $n$-th and $m$-th integrated Brownian motions is given by $\varepsilon^{1-1/(2 \lceil (n+m)/2 \rceil)}$.
Furthermore, we extract a closed-form formula for the squared Wasserstein distance $\c{W}_{2}^{2}(\pi_{\varepsilon}, \pi_{0})$ between EOT coupling and OT couplings. In finite dimensions $\c{H} = \b{R}^{d}$, this distance typically decays at the parametric rate $\c{O}(\varepsilon)$, whereas infinite-dimensional settings may again exhibit slower, non-parametric rates.

\section{Preliminaries}\label{sec:prelim}

\subsection{Notations}
Let $\c{H}$ denote a real Hilbert space, which is assumed to be separable throughout this work. For any closed subspace $\c{V} \subset \c{H}$, we denote the orthogonal projection onto $\c{V}$ by $\m{\Pi}_{\c{V}}$. The symbol $^{\dagger}$ denotes the Moore-Penrose pseudoinverse. For an operator $\m{T}$, we denote its null space and its range by $\c{N}(\m{T})$ and $\c{R}(\m{T})$, respectively. For a densely defined closable operator $\m{T}$, $\overline{\m{T}}$ denotes its closure \cite{conway2019course}.

The Banach space of bounded operators on $\c{H}$ is denoted by $\c{B}_{\infty}(\c{H})$ and is equipped with the operator norm $\vertiii{\cdot}_{\infty}$. 
% For any $\m{K} \in \c{B}_{\infty}(\c{H})$, the following properties hold: $\c{R}(\m{K}^{\dagger}) = \c{N}(\m{K})^{\perp}$, $\m{K}^{\dagger} \m{K} = \m{\Pi}_{\c{N}(\m{K})^{\perp}}$, and the closure corresponds to $\overline{\m{K} \m{K}^{\dagger}} = \m{\Pi}_{\overline{\c{R}(\m{K})}}$. 
We denote by $\c{B}_{0}(\c{H})$ the Banach space of compact operators. For $1 \le p < \infty$, the Banach space of $p$-Schatten class operators on $\c{H}$ is defined as:
\begin{equation*}
    \c{B}_{p}(\c{H}) := \left\{ \m{K} \in \c{B}_{0}(\c{H}) : \vertiii{\m{K}}_{p} := \left( \sum_{l=1}^{\infty} \innpr{e_{l}}{(\m{K}^{*} \m{K})^{p/2} e_{l}} \right)^{1/p} < +\infty \right\}, 
\end{equation*}
where $\{e_{l}\}_{l=1}^{\infty}$ is any orthonormal basis (ONB) of $\c{H}$. We denote by $\c{B}_{\infty}^{+}(\c{H})$ the Banach space of non-negative definite (n.n.d.) bounded operators, and similarly define $\c{B}_{0}^{+}(\c{H})$ and $\c{B}_{p}^{+}(\c{H})$ for $p \in [1, \infty)$. For $p=2$, the class of $2$-Schatten operators is referred to as the Hilbert–Schmidt operators. For $p=1$, the class consists of the trace-class operators. For $\m{K} \in \c{B}_{1}(\c{H})$, the trace is defined as
\begin{equation*}
    \tr (\m{K}) := \sum_{l=1}^{\infty} \innpr{e_{l}}{\m{K} e_{l}},
\end{equation*}
which coincides with the trace norm $\vertiii{\cdot}_{1}$ if and only if $\m{K} \in \c{B}_{1}^{+}(\c{H})$.

For the analysis of EOT, we utilize the functional calculus for self-adjoint compact operators \cite{hall2013quantum}. Let $\m{A} \in \c{B}_{0}(\c{H})$ be a self-adjoint compact operator. By the spectral theorem, $\m{A}$ admits a spectral decomposition into distinct eigenvalues $\sigma(\m{A}) = \{\lambda_{k}\}_{k \in \c{I}} \subset \b{R}$ and associated orthogonal projections $\{\m{\Pi}_{k}\}_{k \in \c{I}}$ onto the eigenspaces $\c{N}(\m{A} - \lambda_{k} \m{I})$. The index set $\c{I}$ is a countable subset of $\b{N}$. To explicitly handle infinite-dimensional degeneracy, we adopt the following convention:
\begin{enumerate}[leftmargin = *]
\item \textbf{Injective Case:} If $\c{N}(\m{A}) = \{0\}$, we set $\c{I} = \b{N}$ and order the eigenvalues strictly by magnitude: $|\lambda_{1}| > |\lambda_{2}| > \cdots > 0$.
\item \textbf{Degenerate Case:} If $\c{N}(\m{A}) \neq \{0\}$, we include the zero eigenvalue. We set $\lambda_{1} = 0$, implying $\m{\Pi}_{1} = \m{\Pi}_{\c{N}(\m{A})}$, and order the remaining non-zero eigenvalues as $|\lambda_{2}| > |\lambda_{3}| > \cdots > 0$.
\end{enumerate}
In both cases, the projections form a resolution of the identity, $\sum_{k \in \c{I}} \m{\Pi}_{k} = \m{I}_{\c{H}}$, where the series converges in the Strong Operator Topology (SOT).
For any bounded Borel function $f: \sigma(\m{A}) \to \b{R}$, the functional calculus is defined by the series:
\begin{equation*}
    f(\m{A}) 
    := \sum_{k \in \c{I}} f(\lambda_{k}) \m{\Pi}_{k}.
\end{equation*}
Since the projections are mutually orthogonal, this series always converges in the SOT; that is, $f(\m{A})v = \sum_{k \in \c{I}} f(\lambda_{k}) (\m{\Pi}_{k} v)$ for all $v \in \c{H}$. The resulting operator $f(\m{A})$ is bounded and self-adjoint. Furthermore, if $f$ satisfies the condition $f(0)=0$ and is continuous at $0$, then the coefficients $f(\lambda_k)$ decay to $0$, ensuring that $f(\m{A})$ is compact ($\in \c{B}_{0}(\c{H})$) and the series converges in the operator norm topology $\vertiii{\cdot}_{\infty}$. 
Finally, we recall the following convergence properties: given a sequence of uniformly bounded functions $\{f_{n}: \sigma(\m{A}) \to \b{R}\}$ that converges pointwise to $f$, then the operator sequence converges in SOT, $f_{n}(\m{A}) \xrightarrow{SOT} f(\m{A})$. If the scalar convergence $f_{n} \to f$ is uniform on the spectrum, the operator convergence holds in the norm topology, $f_{n}(\m{A}) \xrightarrow{\vertii{\cdot}_{\infty}} f(\m{A})$.

\subsection{Review of Gaussian OT}\label{ssec:rev:OT}
We briefly review key results from our recent work \cite{yun2025gaussian} that are essential for this paper. Any centered Gaussian measure $\mu = N(\m{0}, \m{A})$ is fully determined by its trace-class covariance $\m{A} \in \c{B}_{1}^{+}(\c{H})$ \cite{da2006introduction}.

\begin{definition}[Partial Isometry]
An operator $\m{U} \in \c{B}_{\infty}(\c{H})$ is called a \textbf{partial isometry} if $\|\m{U} h \| = \|h \|$ for any $h \in \c{N}(\m{U})^{\perp}$. The subspaces $\c{N}(\m{U})^{\perp}$ and $\c{R}(\m{U})$ are called the \emph{initial space} and the \emph{final space}, respectively.
\end{definition}

The final space of a partial isometry $\m{U}$ is closed. The operators $\m{U}^{*} \m{U}$ and $\m{U} \m{U}^{*}$ are the orthogonal projections onto the initial and final spaces, respectively. Furthermore, the adjoint $\m{U}^{*} = \m{U}^{\dagger}$ is also a partial isometry with the initial and final spaces swapped \cite{conway2019course}.

\begin{definition}[Green's Operator]\label{def:Green's}
Let $\m{A} \in \c{B}_{1}^{+}(\c{H})$. An operator $\m{G} \in \c{B}_{2}(\c{H})$ satisfying $\m{G} \m{G}^{*} = \m{A}$ is called a \textbf{Green's operator} for $\m{A}$. We denote the set of such operators by $\c{G}(\m{A})$.
\end{definition}
For any Green's operator $\m{G} \in \c{G}(\m{A})$, $\m{U}_{\m{G}} = \m{A}^{\dagger/2} \m{G}$ is the partial isometry with the initial space $\c{N}(\m{U}_{\m{G}})^{\perp} = \c{N}(\m{G})^{\perp} = \overline{\c{R}(\m{G}^{*})}$ and the final space $\c{R}(\m{U}_{\m{G}}) = \overline{\c{R}(\m{G})} = \overline{\c{R}(\m{A})}$.
Consider the orthogonal decomposition of $\c{H}$ induced by $\m{A}$:
\begin{equation*}
    \c{H} = \c{H}_{1} \oplus \c{H}_{2}, \quad \c{H}_{1} := \overline{\c{R}(\m{A})}, \quad \c{H}_{2} := \c{N}(\m{A}).
\end{equation*}
We remark that the operator $\m{B}^{1/2} \m{\Pi}_{\c{N}(\m{A}^{1/2} \m{B}^{1/2})} \m{B}^{1/2}$ vanishes on $\c{H}_{1}$, and $\c{H}_{2}$ is the invariant subspace.
\begin{definition}[Schur Operator]
Let $\m{A}, \m{B} \in \c{B}_{1}^{+}(\c{H})$. We define the $\m{A}$-\textbf{Schur operator} of $\m{B}$ as
\begin{align*}
    \m{B}/\m{A} := [\m{B}^{1/2} \m{\Pi}_{\c{N}(\m{A}^{1/2} \m{B}^{1/2})} \m{B}^{1/2}] \big|_{\c{H}_{2}} : \c{H}_{2} \to \c{H}_{2}.
\end{align*}
\end{definition}

To interpret the Schur operator, express $\m{A}$ and $\m{B}$ as block operators:
\begin{align*}
    \m{A} =
    \begin{pmatrix}
    \m{A}_{{11}} & \m{0} \\
    \m{0} & \m{0}
    \end{pmatrix},
    \quad
    \m{B} =
    \begin{pmatrix}
    \m{B}_{{11}} & \m{B}_{{12}} \\
    \m{B}_{{21}} & \m{B}_{{22}}
    \end{pmatrix}
        :
    \begin{array}{c}
    \c{H}_{1} \\
    \oplus \\
    \c{H}_{2}
    \end{array}
    \rightarrow
    \begin{array}{c}
    \c{H}_{1} \\
    \oplus \\
    \c{H}_{2}
    \end{array},
\end{align*}
where $\m{A}_{{11}} := \m{A} \big|_{\c{H}_{1}}$ is injective, and the block components of $\m{B}$ are given by $\m{B}_{{ij}} := \m{\Pi}_{\c{H}_{i}} \m{B} \big|_{\c{H}_{j}}: \c{H}_{j} \to \c{H}_{i}$ for $i, j \in \{1, 2\}$. It can be shown that $\m{B} / \m{A} = \m{B}_{{22}} - (\m{B}_{{11}}^{\dagger/2} \m{B}_{{12}})^{*} (\m{B}_{{11}}^{\dagger/2} \m{B}_{{12}}): \c{H}_{2} \to \c{H}_{2}$. The following conditions are equivalent:
\begin{align*}
    \m{B}/\m{A} = \m{0} \quad \Longleftrightarrow \quad \c{N} (\m{B}^{1/2}) = \c{N}(\m{A}^{1/2} \m{B}^{1/2}) \quad \Longleftrightarrow \quad \c{R}(\m{B}^{1/2}) \cap \c{N}(\m{A}^{1/2}) = \{ \m{0} \}.
\end{align*}

\begin{definition}[Proper Alignment]
Let $\m{A}, \m{B} \in \c{B}_{1}^{+}(\c{H})$. We say that $\m{G} \in \c{G}(\m{A})$ and $\m{M} \in \c{G}(\m{B})$ are \textbf{properly aligned} if $\m{G}^{*} \m{M} \in \c{B}_{1}^{+}(\c{H})$.
\end{definition}

By von Neumann's trace inequality and the polar factorization \cite{conway2019course}, we have
\begin{align}\label{eq:W2:bdd:HS}
    \tr [\m{G}^{*} \m{M}] \le \tr [(\m{G}^{*} \m{B} \m{G})^{1/2}] = \tr [(\m{A}^{1/2} \m{B} \m{A}^{1/2})^{1/2}] \quad \Longleftrightarrow \quad \c{W}_{2}^{2}(\m{A}, \m{B}) \le \vertii{\m{G} - \m{M}}_{2}^{2},
\end{align}
and the equality holds if and only if $\m{G}$ and $\m{M}$ are properly aligned. In this case, the eigenvalues of $\m{G}^{*} \m{M} = \m{G}^{*} \m{M} \in \c{B}_{1}^{+}(\c{H})$ coincides with that of  $(\m{A}^{1/2} \m{B} \m{A}^{1/2})^{1/2}  \in \c{B}_{1}^{+}(\c{H})$.
We also remark that if $\m{A}$ and $\m{B}$ commute, then $\m{A}^{1/2}$ and $\m{B}^{1/2}$ are properly aligned. Additionally, if $\m{A} \m{B} = \m{0}$, then $\m{G}^{*} \m{M} = \m{0}$ for any $\m{G} \in \c{G}(\m{A})$ and $\m{M} \in \c{G}(\m{B})$; hence, they are properly aligned. 
% Another characterization of the Bures-Wasserstein distance is given by
A coupling of $N(\m{0}, \m{A})$ and $N(\m{0}, \m{B})$ is optimal for the Kantorovich problem if and only if its covariance operator admits the form
\begin{align}\label{eq:Kanto:block}
    \left(
    \begin{array}{c|c}
    \m{A} & \m{G} \m{M}^{*} \\
    \hline
    \m{M} \m{G}^{*} & \m{B}
    \end{array}
    \right) \in \c{B}_{1}^{+}(\c{H} \times \c{H}),
\end{align}
for some properly aligned Green's operators $\m{G} \in \c{G}(\m{A})$ and $\m{M} \in \c{G}(\m{B})$. Throughout this work, we employ grid lines in block operator representations on $\c{H} \times \c{H}$, to visually distinguish from the decomposition $\c{H} = \c{H}_{1} \oplus \c{H}_{2}$.

\begin{definition}[Reachability]
Let $\m{A}, \m{B} \in \c{B}_{1}^{+}(\c{H})$. We say that $\m{B}$ is \textbf{reachable} from $\m{A}$, denoted by $\m{A} \leadsto \m{B}$, if there exists a (possibly unbounded) operator $\m{T}: \c{R}(\m{A}^{1/2}) \subset \c{H}_{1} \rightarrow \c{H}$ satisfying transportability and optimality conditions:
\begin{align*}
    (\m{T} \m{A}^{1/2}) (\m{T} \m{A}^{1/2})^{*} = \m{B}, \quad \tr (\m{T} \m{A}) = \tr [(\m{A}^{1/2} \m{B} \m{A}^{1/2})^{1/2}].
\end{align*}
\end{definition}

In such cases, there exists an extension $\tilde{\m{T}}$ of $\m{T}$ that optimally pushes forward $\mu = N(\m{0}, \m{A})$ to $\nu = N(\m{0}, \m{B})$, satisfying:
\begin{align*}
    \c{W}_{2}^{2}(\m{A}, \m{B}) = \int_{\c{H} \times \c{H}} \|x-\tilde{\m{T}}(x)\|^{2} \mu(\rd x) = \tr [\m{A}] + \tr [\m{B}] - 2 \tr [(\m{A}^{1/2} \m{B} \m{A}^{1/2})^{1/2}].
\end{align*}
It has been established in \cite{yun2025gaussian} that for arbitrary $\m{A}, \m{B} \in \c{B}_{1}^{+}(\c{H})$, $\m{A} \leadsto \m{B}$ holds if and only if
\begin{align*}
    \dim [\c{R}(\m{B}^{1/2}) \cap \c{H}_{2}] \le \dim [\c{N}(\m{B}^{1/2} \m{A}^{1/2}) \cap \c{H}_{1}],
\end{align*}
and thus either $\m{A} \leadsto \m{B}$ or $\m{B} \leadsto \m{A}$; that is, a Monge solution always exists in at least one direction. In the finite dimensional case $\c{H} = \b{R}^{d}$, the rank nullity theorem reduces the above condition to $\mathrm{rank}(\m{A}) \ge \mathrm{rank}(\m{B})$, and in this case, there always exists an optimal linear transport map $\m{T} \in \b{R}^{d \times d}$.

\begin{theorem}[Kantorovich Problem]\label{thm:Kanto:char}
Let $\m{A}, \m{B} \in \c{B}_{1}^{+}(\c{H})$ be covariance operators such that $\m{A} \leadsto \m{B}$. Fix $\m{G} = \m{A}^{1/2} \in \c{G}(\m{A})$.
Then, the set of properly aligned Green's operators $\m{M} \in \c{G}(\m{B})$ is non-empty and is characterized by operators of the form:
\begin{align*}
    \m{M} =
    \begin{pmatrix}
    \m{A}_{{11}}^{\dagger/2} (\m{A}_{{11}}^{1/2} \m{B}_{{11}} \m{A}_{{11}}^{1/2})^{1/2} & \m{0} \\
    [(\m{A}_{{11}}^{1/2} \m{B}_{{11}} \m{A}_{{11}}^{1/2})^{\dagger/2} \m{A}_{{11}}^{1/2} \m{B}_{{12}} + \m{N}_{{12}}]^{*} & \m{M}_{{22}}
    \end{pmatrix}
        :
    \begin{array}{c}
    \c{H}_{1} \\
    \oplus \\
    \c{H}_{2}
    \end{array}
    \rightarrow
    \begin{array}{c}
    \c{H}_{1} \\
    \oplus \\
    \c{H}_{2}
    \end{array} ,
\end{align*}
where: 
\begin{itemize}[leftmargin = *]
    \item $\m{N}_{{12}}  : \c{H}_{2} \to \c{H}_{1}$ satisfies $\c{R}(\m{N}_{{12}}) \subset \c{N}(\m{B}_{{11}}^{1/2} \m{A}_{{11}}^{1/2})$ and $\m{N}_{{12}}^{*} \m{N}_{{12}} \preceq \m{B}/\m{A}$.
    \item $\m{M}_{{22}}  : \c{H}_{2} \to \c{H}_{2}$ satisfies $\m{M}_{{22}} \m{M}_{{22}}^{*} = \m{B}/\m{A} - \m{N}_{{12}}^{*} \m{N}_{{12}}$.
\end{itemize}
A corresponding optimal prepushforward exists (i.e., $\m{M} = \m{T}\m{A}^{1/2}$ for some linear operator $\m{T}: \c{R}(\m{A}^{1/2}) \subset \c{H}_{1} \rightarrow \c{H}$) if and only if one selects $\m{M}_{{22}} = \m{0}$, which implies $\m{N}_{{12}}^{*} \m{N}_{{12}} = \m{B}/\m{A}$.
\end{theorem}
The optimal coupling in \eqref{eq:Kanto:block} depends solely on the choice of $\m{N}_{{12}}$, and is independent of $\m{M}_{{12}}$:
\begin{small}
\begin{align*}
    \m{G} \m{M}^{*} = \begin{pmatrix}
    \m{A}_{{11}}^{1/2} [\m{A}_{{11}}^{\dagger/2} (\m{A}_{{11}}^{1/2} \m{B}_{{11}} \m{A}_{{11}}^{1/2})^{1/2}]^{*}  & \m{A}_{{11}}^{1/2} (\m{A}_{{11}}^{1/2} \m{B}_{{11}} \m{A}_{{11}}^{1/2})^{\dagger/2} \m{A}_{{11}}^{1/2} \m{B}_{{12}}  + \m{A}_{{11}}^{1/2} \m{N}_{{12}} \\
    \m{0} & \m{0}
    \end{pmatrix} :
    \begin{array}{c}
    \c{H}_{1} \\
    \oplus \\
    \c{H}_{2}
    \end{array}
    \rightarrow
    \c{H}_{1}.
\end{align*}
\end{small}
Furthermore, the optimal Gaussian coupling is unique if and only if the $\m{A}$-Schur complement vanishes. 
In this case, the unique optimal prepushforward from $\m{A}$ to $\m{B}$ is given by
\begin{align*}
    \m{T} =
    \begin{pmatrix}
    \m{A}_{{11}}^{\dagger/2} (\m{A}_{{11}}^{1/2} \m{B}_{{11}} \m{A}_{{11}}^{1/2})^{1/2} \m{A}_{{11}}^{\dagger/2} \\
    ((\m{A}_{{11}}^{1/2} \m{B}_{{11}} \m{A}_{{11}}^{1/2})^{\dagger/2} \m{A}_{{11}}^{1/2} \m{B}_{{12}})^{*} \m{A}_{{11}}^{\dagger/2}
    \end{pmatrix}
        :
    \c{R}(\m{A}^{1/2}) \subset \c{H}_{1}
    \rightarrow
    \begin{array}{c}
    \c{H}_{1} \\
    \oplus \\
    \c{H}_{2}
    \end{array}.
\end{align*}

\section{Gaussian EOT}
We formally introduce the Entropic Optimal Transport (EOT) problem in the Gaussian setting. Given two centered Gaussian measures $\mu = \c{N}(\m{0}, \m{A})$ and $\nu = \c{N}(\m{0}, \m{B})$ on $\c{H}$, the EOT problem with regularization parameter $\varepsilon > 0$ with respect to mutual information is defined as:
\begin{align*}
    \c{W}_{2, \varepsilon} (\m{A}, \m{B}) := \left[ \inf_{\pi \in \Pi(\mu, \nu)} \left\{ \int_{\c{H} \times \c{H}} \|x-y\|^{2} \pi(\rd x \rd y) + 2 \varepsilon \KL(\pi \| \pi_{\indep}) \right\} \right]^{1/2},
\end{align*}
where the reference measure is the independent coupling:
\begin{align*}
    \pi_{\indep} := \mu \otimes \nu = \c{N}(\m{0}, \m{\Sigma}_{\indep}), \quad \m{\Sigma}_{\indep} := \begin{pmatrix} \m{A} & \m{0} \\ \m{0} & \m{B} \end{pmatrix} \in \c{B}_{1}^{+}(\c{H} \times \c{H}).
\end{align*}
While the definition allows $\pi$ to be any probability measure, we can restrict our search to Gaussian couplings without loss of generality due to the following argument.
Let $\pi \in \Pi(\mu, \nu)$ be any feasible coupling with finite second moments, and let $\pi_{G}$ be the centered Gaussian measure with the same covariance operator as $\pi$. Since the quadratic cost function depends solely on the second moments, $\pi$ and $\pi_{G}$ yield identical transport costs. 
Furthermore, since the reference measure $\pi_{\indep}$ is Gaussian, the KL divergence satisfies the well-known Pythagorean identity \cite{csiszar1975divergence}: $\KL(\pi \| \pi_{\indep}) = \KL(\pi \| \pi_{G}) + \KL(\pi_{G} \| \pi_{\indep})$.
Since $\KL(\pi \| \pi_{G}) \ge 0$ with equality if and only if $\pi = \pi_{G}$, the Gaussian coupling $\pi_{G}$ strictly dominates any non-Gaussian candidate with the same covariance. 
Finally, the uniqueness of the solution is guaranteed by the strict convexity of the mapping $\pi \mapsto \KL(\pi \| \pi_{\indep})$ on the convex set of couplings $\Pi(\mu, \nu)$, combined with the linearity of the transport cost.
Thus, we assume $\pi = \c{N}(\m{0}, \m{\Sigma})$ where $\m{\Sigma}$ has marginals $\m{A}$ and $\m{B}$.
For the entropic cost to be finite, the coupling $\pi$ must be absolutely continuous with respect to $\pi_{\indep}$ so that $\KL(\pi \| \pi_{\indep}) < \infty$. Our first task is to completely characterize the geometry of such Gaussian couplings.

\subsection{Admissible Correlations}

\begin{proposition}[Generalized Baker-Douglas Theorem \cite{douglas1966majorization, baker1970mutual}]
Let $\c{H}$ be a separable Hilbert space and $\m{A}, \m{B} \in \c{B}_{1}^{+}(\c{H})$ be trace-class covariance operators. Fix any Green's operators $\m{G} \in \c{G}(\m{A})$ and $\m{M} \in \c{G}(\m{B})$.
The Gaussian coupling of $\c{N}(\m{0}, \m{A})$ and $\c{N}(\m{0}, \m{B})$ on $\c{H} \times \c{H}$ admits a covariance of the form:
\begin{align*}
    \m{\Sigma} = 
    \left(
    \begin{array}{c|c}
    \m{A} & \m{C} \\
    \hline
    \m{C}^{*} & \m{B}
    \end{array}
    \right) \in \c{B}_{1}^{+}(\c{H} \times \c{H}),
\end{align*}
for some shrinkage $\m{R} \in \c{B}_{\infty}(\c{H})$ (i.e., $\vertiii{\m{R}}_{\infty} \le 1$) such that the cross-covariance is given by $\m{C} = \m{G} \m{R} \m{M}^*$. This operator $\m{R}$ is unique if one imposes the following constraints, in which case we call it the \textbf{Green's correlation operator}:
\begin{align*}
    \c{N}(\m{R}) \supset \c{N}(\m{M}) \quad \text{and} \quad \c{R}(\m{R}) \subset \c{N}(\m{G})^{\perp}.
\end{align*}
Explicitly, it is given by $\m{R} = \left[ \m{M}^{\dagger} (\m{G}^{\dagger} \m{C})^{*} \right]^{*} = \overline{\m{G}^{\dagger} \m{C} (\m{M}^{*})^{\dagger}}$.
\end{proposition}
\begin{proof}
The condition $\m{\Sigma} \succeq \m{0}$ implies that for any $u, v \in \c{H}$, $|\langle u, \m{C} v \rangle|^2 \le \langle u, \m{A} u \rangle \langle v, \m{B} v \rangle$. This inequality necessitates the range inclusions $\c{R}(\m{C}) \subset \c{R}(\m{A}^{1/2})$ and $\c{R}(\m{C}^*) \subset \c{R}(\m{B}^{1/2})$. 
Since $\c{R}(\m{A}^{1/2}) = \c{R}(\m{G})$ and $\c{R}(\m{B}^{1/2}) = \c{R}(\m{M}) = \c{R}(\m{M}^{\dagger *})$, the Douglas Factorization Lemma guarantees the existence of a bounded operator $\m{R}$ such that $\m{C} = \m{G} \m{R} \m{M}^{*}$. 
Furthermore, the Douglas lemma states that among all such operators, there exists a unique solution satisfying the geometric constraints $\c{N}(\m{R}) \supset \c{N}(\m{M})$ and $\c{R}(\m{R}) \subset \c{N}(\m{G})^{\perp}$. This specific solution coincides with the Moore-Penrose formulation $\overline{\m{G}^{\dagger} \m{C} (\m{M}^{*})^{\dagger}}$. Additionally, it possesses the minimal operator norm among all valid factors; since there exists at least one contraction satisfying the factorization (by the Cauchy-Schwarz argument on $\m{\Sigma} \succeq \m{0}$):
\begin{align*}
    |\innpr{(\m{G}^{*} u)}{\m{R}(\m{M}^{*} v)}|^{2} = |\langle u, \m{C} v \rangle|^2 \le \| \m{G}^{*} u\|^{2} \|\m{M}^{*} v\|^{2},
\end{align*}
this unique minimal-norm solution must satisfy $\vertiii{\m{R}}_{\infty} \le 1$.
\end{proof}

\begin{theorem}[Feldman-Hájek Dichotomy for Couplings \cite{da2006introduction}]\label{thm:feld:coup}
Let $\m{A}, \m{B} \in \c{B}_{1}^{+}(\c{H})$ and fix any Green's operators $\m{G} \in \c{G}(\m{A})$ and $\m{M} \in \c{G}(\m{B})$. Let $\pi$ be a centered Gaussian measure on $\c{H} \times \c{H}$ defined by the covariance operator:
\begin{align*}
    \m{\Sigma} = \left(
    \begin{array}{c|c}
    \m{A} & \m{G} \m{R} \m{M}^* \\
    \hline
    \m{M} \m{R}^* \m{G}^* & \m{B}
    \end{array}
    \right) \in \c{B}_{1}^{+}(\c{H} \times \c{H}),
\end{align*}
where $\m{R}$ is the Green's correlation operator for $\m{\Sigma}$. Then, $\pi$ is equivalent to $\pi_{\indep}$ if and only if
\begin{align}\label{eq:FH:dicho}
     \m{R} \in \c{B}_{2}(\c{H}) \quad \text{and} \quad \vertiii{\m{R}}_{\infty} < 1;
\end{align}
otherwise, they are mutually singular.
\end{theorem}
\begin{proof}
By the classical Feldman-Hájek theorem, $\pi$ and $\pi_{\indep}$ are equivalent if and only if:
\begin{enumerate}[leftmargin = *]
    \item \textbf{Range Equivalence:} $\c{R}(\m{\Sigma}^{1/2}) = \c{R}(\m{\Sigma}_{\indep}^{1/2})$.
    \item \textbf{Hilbert-Schmidt Defect:} The operator $\m{T} = \m{\Sigma}_{\indep}^{\dagger/2} (\m{\Sigma} - \m{\Sigma}_{\indep}) \m{\Sigma}_{\indep}^{\dagger/2}$ is Hilbert-Schmidt on the Cameron-Martin space of $\pi_{\indep}$, and $\m{I} + \m{T}$ is invertible,
\end{enumerate}
otherwise mutually singular.

\noindent \textbf{Step 1: Range Equivalence ($ \Longleftrightarrow \ \vertiii{\m{R}}_{\infty} < 1$).} \\
Let $\m{S} = \text{diag}(\m{G}, \m{M})$. We observe that $\m{\Sigma}_{\indep} = \m{S}\m{S}^*$. We can factor the coupled covariance $\m{\Sigma}$ similarly:
\begin{align*}
    \m{\Sigma} = \m{S} \underbrace{\left(
    \begin{array}{c|c}
    \m{I} & \m{R} \\
    \hline
    \m{R}^* & \m{I}
    \end{array}
    \right)}_{\m{K}} \m{S}^*.
\end{align*}
Since $\c{R}(\m{\Sigma}_{\indep}^{1/2}) = \c{R}(\m{S})$, the ranges coincide if and only if the core matrix $\m{K}$ is a bounded invertible operator on the support of $\pi_{\indep}$.
Using Schur complements, $\m{K} \succ \m{0}$ if and only if $\m{I} - \m{R}^* \m{R} \succ \m{0}$. This condition is equivalent to the spectrum of $\m{R}^* \m{R}$ being strictly bounded below $1$, which implies $\vertiii{\m{R}}_{\infty} < 1$. Note that this strict inequality automatically guarantees that $\m{K}$ (and thus $\m{I}+\m{T}$) is invertible.

\noindent \textbf{Step 2: Hilbert-Schmidt Condition ($\Longleftrightarrow \ \m{R} \in \c{B}_{2}(\c{H})$).} \\
We examine the relative covariance operator $\m{T}$. Since the measures may be degenerate:
\begin{align*}
    \renewcommand{\arraystretch}{1.3}
    \m{T} &= 
    \left(
    \begin{array}{c|c}
    \m{A}_1^{\dagger/2} & \m{0} \\
    \hline
    \m{0} & \m{B}^{\dagger/2}
    \end{array}
    \right)
    \left(
    \begin{array}{c|c}
    \m{0} & \m{G} \m{R} \m{M}^* \\
    \hline
    \m{M} \m{R}^* \m{G}^* & \m{0}
    \end{array}
    \right)
    \left(
    \begin{array}{c|c}
    \m{A}_1^{\dagger/2} & \m{0} \\
    \hline
    \m{0} & \m{B}^{\dagger/2}
    \end{array}
    \right) \\
    &= 
    \left(
    \begin{array}{c|c}
    \m{0} & (\m{A}_1^{\dagger/2} \m{G}) \m{R} (\m{B}^{\dagger/2} \m{M})^* \\
    \hline
    (\m{B}^{\dagger/2} \m{M}) \m{R}^* (\m{A}_1^{\dagger/2} \m{G})^{*} & \m{0}
    \end{array}
    \right)
    = 
    \left(
    \begin{array}{c|c}
    \m{0} & \m{U}_{\m{G}} \m{R} \m{U}_{\m{M}}^* \\
    \hline
    \m{U}_{\m{M}} \m{R}^* \m{U}_{\m{G}}^* & \m{0}
    \end{array}
    \right),
\end{align*}
where we invoke the partial isometries defined in the preliminaries: $\m{U}_{\m{G}} := \m{A}^{\dagger/2} \m{G}$ and $\m{U}_{\m{M}}:=\m{B}^{\dagger/2} \m{M}$. Thus, $\m{T} \in \c{B}_2(\c{H} \times \c{H})$ if and only if its off-diagonal block $\tilde{\m{R}} = \m{U}_{\m{G}} \m{R} \m{U}_{\m{M}}^* \in \c{B}_{2}(\c{H})$ is Hilbert-Schmidt.
The partial isometry $\m{U}_{\m{M}}^*$ maps $\c{N}(\m{B})^\perp$ isometrically to $\c{N}(\m{M})^\perp$. The operator $\m{R}$ maps $\c{N}(\m{M})^\perp$ to $\c{N}(\m{G})^\perp$. Finally, $\m{U}_{\m{G}}$ maps $\c{N}(\m{G})^\perp$ isometrically to $\c{N}(\m{A})^\perp$.
Since the partial isometries are isometric on the support and range of $\m{R}$ (due to the geometric constraints of the Green's correlation operator), they preserve the Hilbert-Schmidt norms: $\vertiii{\m{U}_{\m{G}} \m{R} \m{U}_{\m{M}}^*}_{2} = \vertiii{\m{R}}_{2}$.
Therefore, $\m{T} \in \c{B}_2(\c{H} \times \c{H})$ if and only if $\m{R} \in \c{B}_2(\c{H})$.
\end{proof}

As a direct consequence, it suffices for the Gaussian EOT to consider the couplings where the Green's correlation operator lies within the following convex set:
\begin{align*}
    \c{F}(\m{G}, \m{M}) := \left\{ \m{R} \in \c{B}_{2}(\c{H}) \;\big|\; \c{N}(\m{R}) \supset \c{N}(\m{M}), \, \c{R}(\m{R}) \subset \c{N}(\m{G})^{\perp}, \, \vertiii{\m{R}}_{\infty} < 1 \right\}. 
\end{align*}

\begin{proposition}
Let $\m{A}, \m{B} \in \c{B}_{1}^{+}(\c{H})$ and fix any Green's operators $\m{G} \in \c{G}(\m{A})$ and $\m{M} \in \c{G}(\m{B})$. Let $\pi$ be a centered Gaussian coupling of $\mu = \c{N}(\m{0}, \m{A})$ and $\nu = \c{N}(\m{0}, \m{B})$ on $\c{H} \times \c{H}$, and let $\m{R}$ be the Green's correlation operator for its covariance. Then, the KL divergence is given by:
\begin{align*}
    \KL(\pi \| \pi_{\indep}) = 
    \begin{cases}
        -\frac{1}{2} \log \det(\m{I} - \m{R}^* \m{R}) &, \quad \m{R} \in \c{F}(\m{G}, \m{M}), \\
        +\infty &, \quad \text{otherwise},
    \end{cases}
\end{align*}
where $\det$ denotes the standard Fredholm determinant for trace-class perturbations \cite{gohberg2012traces}.
\end{proposition}
\begin{proof}
From \cref{thm:feld:coup}, it is immediate that $\KL(\pi \| \pi_{\indep}) = \infty$ if $\m{R} \notin \c{F}(\m{G}, \m{M})$, as the measures are mutually singular. 
Let $\m{R} \in \c{F}(\m{G}, \m{M})$. Consider the operator $\tilde{\m{R}} = \m{U}_{\m{G}} \m{R} \m{U}_{\m{M}}^* \in \c{B}_{2}(\c{H})$ constructed in the proof of \cref{thm:feld:coup}, which shares the same singular values $\{s_k\}_{k=1}^\infty$ as $\m{R}$, satisfying $s_{k} < 1$ and $\sum s_k^2 < \infty$.
We construct a sequence of finite-rank approximations. Let $\tilde{\m{R}}_n$ be the operator formed by the first $n$ singular components of $\tilde{\m{R}}$. We define the approximate Gaussian coupling $\pi_n = \c{N}(\m{0}, \m{\Sigma}_n)$ via the covariance:
\begin{align*}
    \m{\Sigma}_{n} := \left(
    \begin{array}{c|c}
    \m{A} & \m{A}^{1/2} \tilde{\m{R}}_n \m{B}^{1/2} \\
    \hline
    \m{B}^{1/2} \tilde{\m{R}}_n^* \m{A}^{1/2} & \m{B}
    \end{array}
    \right), \quad n \in \b{N}.
\end{align*}
The relative covariance operator $\m{T}_n := \m{\Sigma}_{\indep}^{\dagger/2} (\m{\Sigma}_{n} - \m{\Sigma}_{\indep}) \m{\Sigma}_{\indep}^{\dagger/2}$ takes the form $\left(\begin{smallmatrix} \m{0} & \tilde{\m{R}}_n \\ \tilde{\m{R}}_n^* & \m{0} \end{smallmatrix}\right)$.
Since $\m{T}_n$ is finite-rank, the standard Gaussian KL formula applies. Crucially, the non-zero eigenvalues of $\m{T}_n$ occur in pairs $\{\pm s_k\}_{k=1}^n$, causing the linear trace terms to cancel exactly ($\tr(\m{T}_n) = 0$). Thus:
\begin{align*}
    \KL(\pi_{n} \| \pi_{\indep}) &= \frac{1}{2} \left[ \tr(\m{T}_n) - \log \det(\m{I} + \m{T}_n) \right] \\
    &= -\frac{1}{2} \sum_{k=1}^n \left[ \log(1+s_k) + \log(1-s_k) \right] = -\frac{1}{2} \sum_{k=1}^n \log(1 - s_k^2).
\end{align*}
We now justify the limit $n \to \infty$. 
Since $\tilde{\m{R}}_n \to \tilde{\m{R}}$ in $\c{B}_2(\c{H})$, it follows that $\tilde{\m{R}}_n^* \tilde{\m{R}}_n \to \tilde{\m{R}}^* \tilde{\m{R}}$ in $\c{B}_1(\c{H})$. The Fredholm determinant $\m{A} \mapsto \det(\m{I} + \m{A})$ is continuous on $\c{B}_1(\c{H})$ \cite{gohberg2012traces}, so the RHS converges to $-\frac{1}{2} \log \det(\m{I} - \m{R}^* \m{R})$.
For the LHS, we invoke the martingale convergence of the Radon-Nikodym derivatives. The sequence of measures $\{\pi_n\}$ corresponds to the restriction of the full measure $\pi$ to the filtration generated by the first $n$ singular modes. By the monotonicity of the KL divergence and the martingale convergence theorem for densities of equivalent Gaussian measures \cite{bogachev1998gaussian}, we have:
\begin{align*}
    \KL(\pi \| \pi_{\indep}) = \lim_{n \to \infty} \KL(\pi_{n} \| \pi_{\indep}) = -\frac{1}{2} \log \det (\m{I} - \m{R}^* \m{R}).
\end{align*}
\end{proof}

\subsection{Spectral Shrinkage of Correlations}

The Gaussian EOT problem seeks to minimize the regularized transport cost:
\begin{align*}
    \c{W}_{2, \varepsilon} (\m{A}, \m{B})^{2} &= \inf_{\pi \in \Pi(\mu, \nu)} \left\{ \int_{\c{H} \times \c{H}} \|x-y\|^{2} \pi(\rd x \rd y) + 2 \varepsilon \KL(\pi \| \pi_{\indep}) \right\}, \quad \varepsilon > 0.
\end{align*}
Expanding the quadratic cost and the entropy term using the Green's correlation operator $\m{R}$, this is equivalent to maximizing the \emph{entropic profit}:
\begin{align}
    &\m{\Sigma}_{\varepsilon} = 
    \left(
    \begin{array}{c|c}
    \m{A} & \m{C}_{\varepsilon} \\
    \hline
    \m{C}_{\varepsilon}^{*} & \m{B}
    \end{array}
    \right) \in \c{B}_{1}^{+}(\c{H} \times \c{H}), \quad \m{C}_{\varepsilon} = \m{G} \m{R}_{\varepsilon} \m{M}^* \in \c{B}_{1}(\c{H}), \label{eq:opt:cross-cov} \\
    &\m{R}_{\varepsilon} := \argmax_{\m{R} \in \c{F}(\m{G}, \m{M})} \left\{ \underbrace{2 \tr [\m{R} (\m{M}^{*} \m{G})] + \varepsilon \log \det(\m{I} - \m{R}^* \m{R})}_{=:\c{P}_{\varepsilon}(\m{R})} \right\}. \label{eq:opt:corr}
\end{align}

\begin{theorem}\label{thm:spec:contrac}
Let $\varepsilon > 0$, $\m{A}, \m{B} \in \c{B}_{1}^{+}(\c{H})$, and let $\m{G} \in \c{G}(\m{A})$ and $\m{M} \in \c{G}(\m{B})$ be properly aligned Green's operators.
Then, the unique solution to \eqref{eq:opt:corr} is given by the universal spectral shrinkage:
\begin{align*}
    \m{R}_{\varepsilon} = f_{\varepsilon}(\m{G}^{*} \m{M}) \in \c{B}_{1}^{+}(\c{H}), \quad f_{\varepsilon}: x \in [0, \infty) \ \mapsto \ \frac{2x}{\sqrt{4x^{2}+\varepsilon^{2}}+\varepsilon} \in [0, 1).
\end{align*}
\end{theorem}
\begin{proof}
We derive the solution via explicit spectral construction.

\noindent \textbf{Step 1: Diagonalization via Alignment.} 
By the proper alignment condition, $\m{K} := \m{G}^{*} \m{M} = \m{M}^{*} \m{G} \in \c{B}_{1}^{+}(\c{H})$ is non-negative definite. By the Spectral Theorem for compact self-adjoint operators, $\m{K}$ admits the decomposition:
\begin{align*}
    \m{G}^{*} \m{M} = \sum_{j=1}^{\infty} \lambda_{j} (\m{e}_{j} \otimes \m{e}_{j}), \quad \lambda_1 \ge \lambda_2 \ge \dots > 0,
\end{align*}
where $\{\m{e}_j\}_{j=1}^\infty$ forms an orthonormal basis for $\overline{\c{R}(\m{K})} \subset \c{N}(\m{G})^\perp \cap \c{N}(\m{M})^\perp$.

\noindent \textbf{Step 2: Von Neumann Trace Inequality.} 
Consider the singular value decomposition of a candidate Green's correlation operator $\m{R} \in \c{F}(\m{G}, \m{M})$:
\begin{align*}
    \m{R} = \sum_{j=1}^{\infty} s_{j} (\m{g}_{j} \otimes \m{h}_{j}) : \m{h}_{j} \in \c{N}(\m{M})^{\perp} \mapsto s_{j} \m{g}_{j} \in \c{N}(\m{G})^{\perp},
\end{align*}
where $1 > s_{1} \ge s_{2} \ge \cdots > 0$. The objective functional splits into linear and entropic parts:
\begin{align*}
    \c{P}_{\varepsilon}(\m{R}) = 2 \tr(\m{R} \m{M}^* \m{G}) + \varepsilon \tr \log(\m{I} - \m{R}^* \m{R}) = 2 \tr(\m{R} \m{G}^* \m{M}) + \varepsilon \sum_{j=1}^{\infty} \log(1 - s_j^2),
\end{align*}
The entropic term depends strictly on the singular values $\{s_j\}$. By the von Neumann Trace Inequality, the linear term is maximized for fixed singular values if and only if the singular vectors align with the eigenvectors of $\m{G}^*\m{M}$, i.e., $\m{g}_j = \m{h}_j = \m{e}_j$. Thus, the optimal $\m{R}$ must be self-adjoint and share the spectral basis of $\m{G}^* \m{M}$.

\noindent \textbf{Step 3: Scalar Optimization.} 
With the basis fixed, the problem reduces to a scalar optimization for each eigenvalue $\lambda_j$:
\begin{align*}
    \sup_{\{s_j\}} \sum_{j=1}^{\infty} \left[ 2 s_j \lambda_j + \varepsilon \log(1 - s_j^2) \right].
\end{align*}
The first-order optimality condition for the function $g(s) = 2 s \lambda_j + \varepsilon \log(1-s^{2})$ yields:
\begin{align*}
    2 \lambda_j - \frac{2 \varepsilon s}{1-s^{2}} = 0 \quad \Rightarrow \quad \varepsilon s^2 + 2\lambda_j s - \varepsilon = 0.
\end{align*}
The unique positive root is given by $s_j = f_{\varepsilon}(\lambda_j)$. Since $f_{\varepsilon}$ maps $(0, \infty)$ strictly into $(0, 1)$, the singular values satisfy strict contractivity.

\noindent\textbf{Step 4: Verification of Constraints.} 
We verify that the constructed operator $\m{R}_{\varepsilon} = \sum f_\varepsilon(\lambda_j) (\m{e}_j \otimes \m{e}_j)$ belongs to the valid set $\c{F}(\m{G}, \m{M})$. First, the basis vectors $\{\m{e}_j\}$ span $\overline{\c{R}(\m{G}^* \m{M})}$, which is a subspace of $\c{N}(\m{G})^\perp \cap \c{N}(\m{M})^\perp$. Consequently, $\m{R}_{\varepsilon}$ vanishes on $\c{N}(\m{M})$ and maps into $\c{N}(\m{G})^\perp$.
Additionally, we observe that $f_\varepsilon(x) \le \frac{2x}{\varepsilon}$ and monotonely increasing for all $x \ge 0$. Therefore, $s_{j}$ is decreasing, and the trace norm is bounded by:
\begin{align*}
    \vertii{\m{R}_{\varepsilon}}_1 = \sum_{j=1}^\infty f_\varepsilon(\lambda_j) \le \frac{2}{\varepsilon} \sum_{j=1}^\infty \lambda_j = \frac{2}{\varepsilon} \tr(\m{G}^* \m{M}) < \infty.
\end{align*}
Since $\m{R}_{\varepsilon}$ is trace-class, it is necessarily Hilbert-Schmidt ($\m{R}_{\varepsilon} \in \c{B}_2(\c{H})$).
Thus, $\m{R}_{\varepsilon} = f_{\varepsilon}(\m{G}^* \m{M})$ is the unique valid maximizer.
\end{proof}

We emphasize a crucial regularity gain: while the optimization domain $\c{F}(\m{G}, \m{M})$ allows for general Hilbert-Schmidt operators, the optimal solution $\m{R}_{\varepsilon}$ is always an  n.n.d. \textbf{trace-class} operator for any $\varepsilon > 0$:
\begin{align*}
    \vertii{\m{R}_{\varepsilon}}_1 \le \frac{2}{\varepsilon} \vertii{(\m{A}^{1/2} \m{B} \m{A}^{1/2})^{1/2}}_1.
\end{align*}
Also, by symmetry, $\m{R}_{\varepsilon} \in \c{F}(\m{G}, \m{M}) \cap \c{F}(\m{M}, \m{G})$.
The behavior of the spectral shrinkage function $f_{\varepsilon}$ is illustrated below. As $\varepsilon \downarrow 0$, the function converges pointwise to the indicator function of the positive reals:
\begin{align*}
    f_{\varepsilon}(x) = \frac{2x}{\sqrt{4x^{2}+\varepsilon^{2}}+\varepsilon} \qquad \stackrel{\varepsilon \downarrow 0}{\longrightarrow} \qquad
    f_{0}(x) = \begin{cases}
        1 &, \quad x > 0, \\
        0 &, \quad x=0.
    \end{cases}
\end{align*}

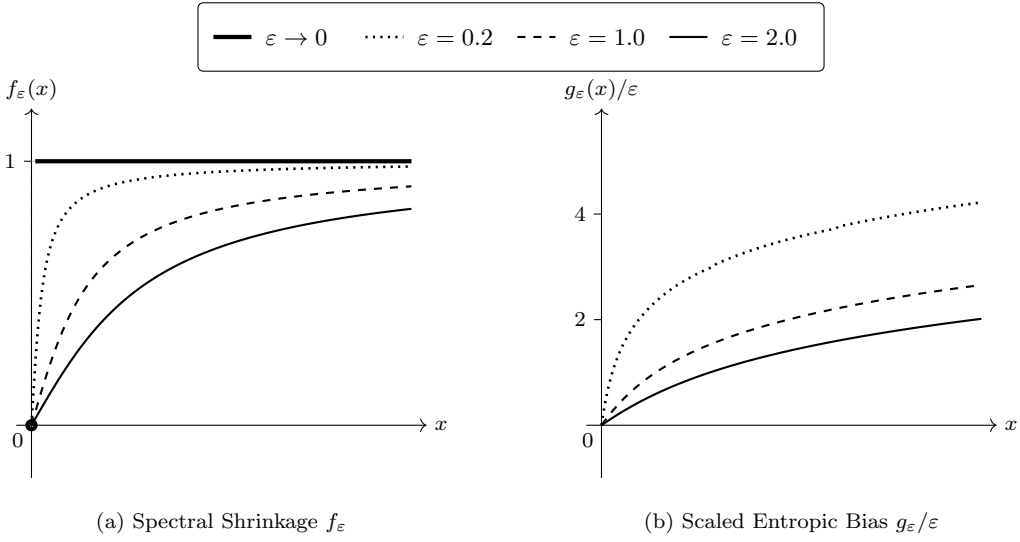
\begin{figure}[h!]
    \centering
    \begin{tikzpicture}
        % --- SHARED DEFINITIONS ---
        \tikzset{
            declare function={
                % Master function f1(y) = 2y / (sqrt(4y^2 + 1) + 1)
                fone(\y) = 2*\y / (sqrt(4*\y*\y + 1) + 1);
                % g(x, eps) = H(x/eps)
                % H(y) = 2y(1 - f1(y)) - ln(1 - f1(y)^2)
                H(\y) = 2*\y * (1 - fone(\y)) - ln(1 - (fone(\y))^2);
            },
            % Style definitions for consistency
            curve20/.style={thick, black},                  % Eps = 2.0
            curve10/.style={thick, black, dashed},          % Eps = 1.0
            curve02/.style={thick, black, dotted, line width=1pt}, % Eps = 0.2
            limit/.style={ultra thick, black}               % Eps -> 0
        }

        % ==========================================
        % LEFT PLOT: f_epsilon (Spectral shrinkage)
        % ==========================================
        % Y-range is roughly 0 to 1. Set unit vector y=3.5cm
        \begin{scope}[local bounding box=leftplot, x=1cm, y=3.5cm]
            
            % Axes
            \draw[->] (-0.2, 0) -- (5.2, 0) node[right] {$x$};
            \draw[->] (0, -0.2) -- (0, 1.2) node[above] {$f_{\varepsilon}(x)$};
            
            % Ticks
            \node at (0,0) [below left] {$0$};
            \draw (0, 1) -- (-0.1, 1) node[left] {$1$};
            \draw[dashed, gray!50] (0,1) -- (5.0,1);

            % Curves
            % Eps = 2.0
            \draw[curve20, domain=0:5.0, samples=100] plot (\x, {2*\x / (sqrt(4*\x*\x + 4) + 2)});
            % Eps = 1.0
            \draw[curve10, domain=0:5.0, samples=100] plot (\x, {2*\x / (sqrt(4*\x*\x + 1) + 1)});
            % Eps = 0.2
            \draw[curve02, domain=0:5.0, samples=200] plot (\x, {2*\x / (sqrt(4*\x*\x + 0.04) + 0.2)});
            
            % Limit case (Eps -> 0)
            \draw[limit] (0,0) circle (1.5pt); 
            \draw[limit] (0.05, 1.0) -- (5.0, 1.0); 
            \draw[dashed, black] (0,0) -- (0,1);

            % Title/Label
            \node[anchor=north] at (2.5, -0.3) {(a) Spectral Shrinkage $f_\varepsilon$};
        \end{scope}

        % ==========================================
        % RIGHT PLOT: g_epsilon (Logarithmic Scale)
        % ==========================================
        % Shift right by 7cm
        % Y-range is roughly 0 to 5. Set unit vector y=0.7cm (so 5*0.7 = 3.5cm height)
        \begin{scope}[shift={(7.5,0)}, local bounding box=rightplot, x=1cm, y=0.7cm]
            
            % Axes
            \draw[->] (-0.2, 0) -- (5.2, 0) node[right] {$x$};
            \draw[->] (0, -1.0) -- (0, 6.0) node[above] {$g_{\varepsilon}(x)/\varepsilon$};
            
            % Ticks
            \node at (0,0) [below left] {$0$};
            \foreach \y in {2, 4} {
                \draw (0, \y) -- (-0.1, \y) node[left] {$\y$};
            }

            % Curves (Using scaling property H(x/eps))
            % Eps = 2.0
            \draw[curve20, domain=0:5.0, samples=100] plot (\x, {H(\x/2.0)});
            % Eps = 1.0
            \draw[curve10, domain=0:5.0, samples=100] plot (\x, {H(\x/1.0)});
            % Eps = 0.2
            \draw[curve02, domain=0:5.0, samples=200] plot (\x, {H(\x/0.2)});

            % Title/Label
            \node[anchor=north] at (2.5, -1.5) {(b) Scaled Entropic Bias $g_\varepsilon/\varepsilon$};
        \end{scope}

        % ==========================================
        % SHARED LEGEND
        % ==========================================
        \node[draw, fill=white, rounded corners=2pt, inner sep=5pt, anchor=south] at (current bounding box.north) {
            \begin{tikzpicture}[baseline, yscale=0.5]
                % Entry 1
                \draw[limit] (0, 0) -- (0.5, 0);
                \node[anchor=west] at (0.5, 0) {\small $\varepsilon \to 0$};
                
                % Entry 2
                \draw[curve02] (2.0, 0) -- (2.5, 0);
                \node[anchor=west] at (2.5, 0) {\small $\varepsilon=0.2$};
                
                % Entry 3
                \draw[curve10] (4.0, 0) -- (4.5, 0);
                \node[anchor=west] at (4.5, 0) {\small $\varepsilon=1.0$};
                
                % Entry 4
                \draw[curve20] (6.0, 0) -- (6.5, 0);
                \node[anchor=west] at (6.5, 0) {\small $\varepsilon=2.0$};
            \end{tikzpicture}
        };

    \end{tikzpicture}
    \caption{(a) The spectral shrinkage function $f_{\varepsilon}(x)$ for varying regularization strengths. As $\varepsilon \to 0$, the function approaches the indicator function $f_{0}(x) = I(x>0)$, forcing the correlation of all positive eigenmodes to 1, whereas larger $\varepsilon$ dampens the correlations, effectively blurring the transport plan. (b) The logarithmic divergence of the scaled entropic bias $g_{\varepsilon}(x)/\varepsilon = g_{1}(x/\varepsilon) \approx \log (x/\varepsilon)$ (see \cref{lem:conv:ftn:mono}) near the origin highlights that the local rate of convergence deteriorates for small singular values. This singularity drives the slower convergence rates in infinite-dimensional settings.}
    \label{fig:spectral_shrinkage}
\end{figure}

\begin{corollary}\label{cor:conv:dist}
Let $\varepsilon > 0$, $\m{A}, \m{B} \in \c{B}_{1}^{+}(\c{H})$. Then,
\begin{align*}
    \c{W}_{2, \varepsilon} (\m{A}, \m{B})^{2} - \c{W}_{2} (\m{A}, \m{B})^{2} = \tr [g_{\varepsilon}((\m{A}^{1/2} \m{B} \m{A}^{1/2})^{1/2})],
\end{align*}
where the entropic bias function is given by $g_{\varepsilon}(x) := \left[ 2 x (1-f_{\varepsilon}(x)) - \varepsilon \log (1- f_{\varepsilon}(x)^{2}) \right]$ for $\varepsilon > 0$ and $x \ge 0$.
In particular:
\begin{enumerate}[leftmargin = *]
\item If $\mathrm{rank} (\m{A}^{1/2} \m{B}^{1/2}) < \infty$, then
\begin{align*}
    \lim_{\varepsilon \downarrow 0} \frac{\c{W}_{2, \varepsilon} (\m{A}, \m{B})^{2} - \c{W}_{2} (\m{A}, \m{B})^{2}}{\varepsilon \log (1/\varepsilon)} = \mathrm{rank} (\m{A}^{1/2} \m{B}^{1/2}).
\end{align*}
\item If the singular values $s_{k}$ of $\m{A}^{1/2} \m{B}^{1/2} \in \c{B}_{1}^{+}(\c{H})$ decay polynomially as $s_{k} = c \, k^{-\alpha}$ for some $c > 0$ and $\alpha > 1$, then the exact asymptotic convergence rate is:
\begin{align*}
    \lim_{\varepsilon \downarrow 0} \frac{\c{W}_{2, \varepsilon} (\m{A}, \m{B})^{2} - \c{W}_{2} (\m{A}, \m{B})^{2}}{\varepsilon^{1-1/\alpha}} = c^{1/\alpha} \int_{0}^{\infty} g_{1} (t^{-\alpha}) \, \rd t \in (0, \infty).
\end{align*}
\end{enumerate}
\end{corollary}
\begin{proof}
The formula follows directly from \cref{thm:spec:contrac} and $\sigma((\m{A}^{1/2} \m{B} \m{A}^{1/2})^{1/2}) = \sigma(\m{G}^* \m{M})$, which are the singular values of $\m{A}^{1/2} \m{B}^{1/2} \in \c{B}_{1}^{+}(\c{H})$. The difference is given by
\begin{align*}
    \c{W}_{2, \varepsilon} (\m{A}, \m{B})^{2} - \c{W}_{2} (\m{A}, \m{B})^{2} = \sum_{k=1}^{\infty} g_{\varepsilon}(s_{k}).
\end{align*}
Using the scaling property $f_{\varepsilon}(x) = f_{1}(x/\varepsilon)$, we can rewrite $g_{\varepsilon}(x) = \varepsilon g_{1}(x/\varepsilon)$, and thus
\begin{align*}
    \frac{\c{W}_{2, \varepsilon} (\m{A}, \m{B})^{2} - \c{W}_{2} (\m{A}, \m{B})^{2}}{\varepsilon} = \sum_{k=1}^{\infty} g_{1} \left( \frac{s_k}{\varepsilon} \right).
\end{align*}
\begin{enumerate}[leftmargin = *]
\item If $R = \mathrm{rank} (\m{A}^{1/2} \m{B} \m{A}^{1/2}) = \mathrm{rank} (\m{A}^{1/2} \m{B}^{1/2}) < \infty$, the sum consists of $R$ terms. As $\varepsilon \downarrow 0$, the argument $y_k = s_k/\varepsilon \to \infty$. By \cref{lem:conv:ftn:mono}, $g_1(y_k) = \log(y_k) + 1 + o(1)$. Therefore:
\begin{align*}
    \sum_{k=1}^{R} g_{1} \left( \frac{s_k}{\varepsilon} \right) = \sum_{k=1}^{R} \left[ \log(s_k) + \log(1/\varepsilon) + 1 + o(1) \right] = R \log(1/\varepsilon) + \sum_{k=1}^{R} \log(s_k) + R + o(1).
\end{align*}
Dividing by $\log(1/\varepsilon)$ and taking the limit yields $R$.

\item If $\mathrm{rank} (\m{A}^{1/2} \m{B} \m{A}^{1/2}) = \infty$ and $s_k = c k^{-\alpha}$, let $u_k := k (\varepsilon/c)^{1/\alpha}$. Then $k = u_k (c/\varepsilon)^{1/\alpha}$ and the mesh size is $\Delta u = (\varepsilon/c)^{1/\alpha}$. 
Substituting $k$ back into the sum, we have:
\begin{align*}
    \frac{\c{W}_{2, \varepsilon}^2 - \c{W}_{2}^2}{\varepsilon^{1 - 1/\alpha}} 
    = \varepsilon^{1/\alpha} \sum_{k=1}^{\infty} g_{1} \left( \left[ k \left( \frac{\varepsilon}{c} \right)^{1/\alpha} \right]^{-\alpha} \right) = c^{1/\alpha} \sum_{k=1}^{\infty} g_{1}(u_k^{-\alpha})  \Delta u.
\end{align*}
Let $h(t) := g_1(t^{-\alpha})$. Since $g_1$ is monotone increasing by \cref{lem:conv:ftn:mono}, $h(t)$ is monotone decreasing. 
\begin{itemize}[leftmargin = *]
    \item \textbf{As $t \to 0$:} The argument $t^{-\alpha} \to \infty$. By \cref{lem:conv:ftn:mono}, $h(t) \sim \log(t^{-\alpha}) = -\alpha \log t$. Since $\int_0^1 |\log t| \rd t < \infty$, $h$ is integrable at 0.
    \item \textbf{As $t \to \infty$:} The argument $t^{-\alpha} \to 0$. Since $h(t) \sim 2 t^{-\alpha}$ by \cref{lem:conv:ftn:mono} and $\alpha > 1$, $\int_1^\infty t^{-\alpha} \rd t < \infty$.
\end{itemize}
Because $h(t)$ is integrable on $(0, \infty)$, the Riemann sum converges to the integral:
\begin{align*}
    \lim_{\varepsilon \downarrow 0} c^{1/\alpha} \sum_{k=1}^{\infty} h(u_k) \Delta u = c^{1/\alpha} \int_{0}^{\infty} g_{1}(t^{-\alpha}) \, \rd t.
\end{align*}
\end{enumerate}
\end{proof}

In the finite-dimensional setting $\c{H} = \b{R}^{d}$, our result recovers the our result recovers the well-known bias scaling $d \varepsilon \log (1/\varepsilon)$ observed in various Gaussian and non-Gaussian setups \cite{pal2019difference, carlier2023convergence, genevay2019sample, janati2020entropic}. In contrast, to the best of our knowledge, the explicit non-parametric rate in a strictly infinite-dimensional setting that links spectral decay to entropic bias is new in the literature.

\begin{example}[Integrated Brownian Motions]
Let $\c{H} = L_{2}[0,1]$ and $n \in \b{N}$. Consider the unbounded differential operator $\m{D}^{n} f = f^{(n)}$ equipped with the boundary conditions $f(0) = f'(0) = \cdots = f^{(n-1)}(0) = 0$. The Green's function $G_{n}: [0, 1] \times [0, 1] \to \b{R}$ is defined as the distributional solution to $\m{D}^{n} G_{n} = \delta$ \cite{hall2013quantum, conway2019course}. This relationship identifies $\m{D}^{n}$ as the pseudoinverse of the integral operator $\m{T}_{G_{n}}$:
\begin{align*}
    \m{T}_{G_{n}} f(s) = \int_{0}^{1} G_{n}(s, t) f(t) \rd t \implies \m{D}^{n} \m{T}_{G_{n}} f(s)= \int_{0}^{1} \delta(s, t) f(t) \rd t = f(s), \quad f \in \c{H}.
\end{align*}
The range of this operator $\c{R}(\m{T}_{G_{n}})$ coincides with the Cameron-Martin space of the $n$-th integrated Brownian motion (IBM) \cite{da2006introduction, hairer2009introduction}. Specifically, the covariance kernel of the $n$-th IBM is given by
\begin{align*}
    K_{n}(s, t) = \int_{0}^{1} G_{n}(s, u) G_{n}(u, t) \rd u, \quad s, t \in [0, 1].
\end{align*}
Consequently, $\text{Dom}(\m{D}^{n}) = \c{R}(\m{G}_{n})$ forms an RKHS with reproducing kernel $K_{n}$. Identifying the Green's operator as $\m{G}_{n} := \m{T}_{G_{n}}$, we observe that $\m{G}_{n} \in \c{G}(\m{A}_{n})$ where $\m{A}_{n} := \m{T}_{K_{n}}$ is the covariance operator. Explicitly, $\m{G}_{n}$ is the $n$-th order Volterra operator \cite{gohberg1978introduction, paulsen2016introduction}:
\begin{align*}
    \m{G}_{n} f (s) = \int_{0}^{s} \frac{(s-t)^{n-1}}{(n-1)!} f(t) \rd t,
\end{align*}
whose singular values exhibit the sharp polynomial decay $s_{k}(\m{G}_{n}) \asymp (\pi k)^{-n}$.
It is established in \cite{yun2025gaussian} that the pair $\m{G}_{n}$ and $\m{G}_{m}$ is properly aligned for any $n, m \in \b{N}$, satisfying the identity:
\begin{align*}
    \innpr{\m{G}_{n} f}{\m{G}_{m} f} = 
    \begin{cases}
        \left\| \m{G}_{(n+m)/2} f \right\|^{2} &, \quad n+m \text{ is even}, \\
        \frac{1}{2} \left\| \m{G}_{(n+m+1)/2} f \right\|^{2} &, \quad n+m \text{ is odd}.
    \end{cases}
\end{align*}
This identity implies that, for $n$-th and $m$-th IBMs, we can set the spectral decay parameter as $\alpha = 2 \lceil (n+m)/2 \rceil$. Then, we apply \cref{cor:conv:dist} to determine the exact asymptotic convergence rate of the entropic bias:
\begin{align*}
    \lim_{\varepsilon \downarrow 0} \frac{\c{W}_{2, \varepsilon} (\m{A}_{n}, \m{A}_{m})^{2} - \c{W}_{2} (\m{A}_{n}, \m{A}_{m})^{2}}{\varepsilon^{1-1/\alpha}} \asymp \int_{0}^{\infty} g_{1} (t^{-\alpha}) \rd t.
\end{align*}
This reveals an infinite-dimensional phenomenon: the convergence rate depends on the smoothness $\alpha$ of the interacting processes.
\end{example}

\begin{corollary}\label{cor:ortho:indep}
Let $\m{A}, \m{B} \in \c{B}_{1}^{+}(\c{H})$. If $\m{A} \m{B} = \m{0}$, $\pi_{\indep}$ is an optimal EOT coupling for any $\varepsilon > 0$.
\end{corollary}
\begin{proof}
By the orthogonality of the ranges ($\c{R}(\m{B}) \subset \c{N}(\m{A}) = \c{R}(\m{A})^\perp$), it holds that $\innpr{\m{G}^* \m{M} x}{y} = \innpr{\m{M} x}{\m{G} y} = 0$ for all $x, y \in \c{H}$. Thus, $\m{G}^* \m{M} = \m{0}$. Since $f_{\varepsilon}(0) = 0$, the optimal correlation is $\m{R}_{\varepsilon} = \m{0}$, yielding the block-diagonal covariance $\m{\Sigma}_{\indep}$.
\end{proof}

Indeed, as the Gaussian EOT problem is strictly convex, the solution should be unique. We formally verify that the derived covariance $\m{\Sigma}_{\varepsilon}$ is intrinsic to the marginals and independent of the choice of Green's operators.

\begin{proposition}[Uniqueness and Invariance of the Coupling]\label{prop:EOT:invar}
Let $\varepsilon > 0$ and $\m{A}, \m{B} \in \c{B}_{1}^{+}(\c{H})$. The covariance $\m{\Sigma}_{\varepsilon}$ defined by \eqref{eq:opt:corr} is independent of the specific choice of properly aligned Green's operators. Consequently, the optimal Gaussian EOT coupling is unique.
\end{proposition}
\begin{proof}
Let $(\m{G}, \m{M})$ and $(\tilde{\m{G}}, \tilde{\m{M}})$ be two pairs of properly aligned Green's operators for $\m{A}$ and $\m{B}$.
Since $\m{G} \m{G}^* = \tilde{\m{G}} \tilde{\m{G}}^* = \m{A}$, the polar decomposition existence theorem implies there exist partial isometries $\m{V}_{\m{A}}, \m{V}_{\m{B}}$ such that $\tilde{\m{G}} = \m{G} \m{V}_{\m{A}}, \tilde{\m{M}} = \m{M} \m{V}_{\m{B}}$, where $\m{V}_{\m{A}}$ maps $\c{N}(\tilde{\m{G}})^{\perp}$ isometrically onto $\c{N}(\m{G})^{\perp}$ (and is zero elsewhere).
Define the aligned products $\m{P} = \m{G}^* \m{M}$ and $\tilde{\m{P}} = \tilde{\m{G}}^* \tilde{\m{M}}$. Substituting the relation yields
\begin{align*}
    \tilde{\m{P}} = \m{V}_{\m{A}}^* \m{G}^* \m{M} \m{V}_{\m{B}} = \m{V}_{\m{A}}^* \m{P} \m{V}_{\m{B}}.
\end{align*}
Since both pairs are properly aligned ($\m{P} \succeq \m{0}$ and $\tilde{\m{P}} \succeq \m{0}$), we can compare their squares. Using $\m{G} \m{G}^* = \m{A}$:
\begin{align*}
    \tilde{\m{P}}^2 = \tilde{\m{P}}^* \tilde{\m{P}} = (\tilde{\m{M}}^* \tilde{\m{G}}) (\tilde{\m{G}}^* \tilde{\m{M}}) = \tilde{\m{M}}^* \m{A} \tilde{\m{M}} = \m{V}_{\m{B}}^* \m{M}^* \m{A} \m{M} \m{V}_{\m{B}} = \m{V}_{\m{B}}^* \m{P}^2 \m{V}_{\m{B}}.
\end{align*}
Since the square root of a positive operator is unique, this implies $\tilde{\m{P}} = \m{V}_{\m{B}}^* \m{P} \m{V}_{\m{B}}$.
Comparing the two expressions 
\begin{align*}
    \m{V}_{\m{A}}^* \m{P} = [\m{V}_{\m{A}}^* \m{P} \m{V}_{\m{B}}] \m{V}_{\m{B}}^* = [\m{V}_{\m{B}}^* \m{P} \m{V}_{\m{B}}] \m{V}_{\m{B}}^* = \m{V}_{\m{B}}^* \m{P} \quad \Rightarrow \quad \m{P} \m{V}_{\m{A}} = \m{P} \m{V}_{\m{B}},
\end{align*}
which implies that $\m{V}_{\m{A}}$ and $\m{V}_{\m{B}}$ act identically on the range $\overline{\c{R}(\m{P})}$.
Now, consider the cross-covariance term $\tilde{\m{C}} = \tilde{\m{G}} f_{\varepsilon}(\tilde{\m{P}}) \tilde{\m{M}}^*$.
We utilize the factorization $f_{\varepsilon}(x) = x g_{\varepsilon}(x)$ where $g_{\varepsilon}(x) = \frac{2}{\sqrt{4x^2+\varepsilon^2}+\varepsilon}$ is bounded on $[0, \infty)$.
\begin{align*}
    \tilde{\m{C}} = \m{G} \m{V}_{\m{A}} \left[ \tilde{\m{P}} g_{\varepsilon}(\tilde{\m{P}}) \right] \m{V}_{\m{B}}^* \m{M}^* = \m{G} \m{V}_{\m{A}} \tilde{\m{P}} \left[ \m{V}_{\m{B}}^* g_{\varepsilon}(\m{P}) \m{V}_{\m{B}} \right] \m{V}_{\m{B}}^* \m{M}^* = \m{G} (\m{V}_{\m{A}} \tilde{\m{P}} \m{V}_{\m{B}}^*) g_{\varepsilon}(\m{P}) \m{M}^*.
\end{align*}
Using $\tilde{\m{P}} = \m{V}_{\m{A}}^* \m{P} \m{V}_{\m{B}}$, we have $\m{V}_{\m{A}} \tilde{\m{P}} \m{V}_{\m{B}}^* = \m{V}_{\m{A}} \m{V}_{\m{A}}^* \m{P} \m{V}_{\m{B}} \m{V}_{\m{B}}^* = \m{\Pi}_{\c{N}(\m{G})^\perp} \m{P} \m{\Pi}_{\c{N}(\m{M})^\perp} = \m{P}$. Thus:
\begin{align*}
    \tilde{\m{C}} = \m{G} \m{P} g_{\varepsilon}(\m{P}) \m{M}^* = \m{G} f_{\varepsilon}(\m{P}) \m{M}^* = \m{C},
\end{align*}
proving that the covariance is invariant.
\end{proof}

\begin{remark}[The Singularity of $\varepsilon \to 0$]
The failure of the proof for $\varepsilon=0$ reflects the fundamental non-uniqueness of the unregularized problem.
The proof relies on the factorization $f_{\varepsilon}(\m{P}) = \m{P} g_{\varepsilon}(\m{P})$ to \emph{absorb} the ambiguity of the partial isometries (specifically, that $\m{V}_{\m{A}}$ and $\m{V}_{\m{B}}$ must agree on the range of $\m{P}$).
When $\varepsilon \to 0$, the function $g_{\varepsilon}(x) \to 1/x$ becomes singular at the origin. Consequently, the operator $g_{0}(\m{P})$ behaves like the pseudoinverse $\m{P}^{\dagger}$, which is unbounded if $\m{P} \in \c{B}_{1}^{+}(\c{H})$ has an infinite rank.
Geometrically, for $\varepsilon > 0$, the entropic penalty forces the correlations to vanish smoothly on the null space of $\m{P}$. For $\varepsilon = 0$, $f_0$ acts as the hard projection $\m{\Pi}_{\overline{\c{R}(\m{P})}}$. This projection preserves any arbitrary unitary rotations on the null space $\c{N}(\m{P})$ -- rotations which are \textbf{not} constrained by the alignment condition $\m{G}^* \m{M} \succeq \m{0}$. Without the entropic \emph{filter}, these ghost rotations persist, leading to non-unique optimal couplings.
\end{remark}

We emphasize that the spectral shrinkage framework established in \cref{thm:spec:contrac} and \cref{prop:EOT:invar} offers more than a theoretical reformulation; it enables a fundamental algorithmic shift. While standard Gaussian EOT solvers rely on Sinkhorn iterations—even when implemented via efficient matrix updates \cite{janati2020entropic, mallasto2022entropy}—they remain iterative fixed-point schemes, susceptible to convergence errors and numerical instability, particularly in the small-regularization regime. In contrast, our approach reduces the Gaussian EOT problem to a direct algebraic computation. By recognizing that the optimal cost is determined purely by the singular spectrum of the operator product $\m{G}^{*}\m{M}$, the exact solution can be recovered using standard linear algebra routines (e.g., SVD of arbitrary Green's operators) in finite time, completely eliminating the need for iterative approximation. This advantage is particularly pronounced in multi-scale simulations or regularization path analysis. Once the eigendecomposition of $\m{G}^{*}\m{M}$ is computed for \textit{any} properly aligned initial pair $(\m{G}, \m{M})$—a single $\c{O}(d^{3})$ operation (e.g., one-step update in \cref{rmk:one:update})—the entropic cost or coupling for any $\varepsilon > 0$ can be evaluated in $\c{O}(d)$ time via scalar spectral mapping. This can yield a dramatic speedup compared to restarting Sinkhorn for each $\varepsilon$, where the required number of iterations typically grows as $\varepsilon \downarrow 0$ \cite{cuturi2013sinkhorn, ghosal2025convergence}.

\subsection{Convergence of Couplings}
Following \cref{ssec:rev:OT}, we now assume that $\m{B}$ is reachable from $\m{A}$, i.e., $\m{A} \leadsto \m{B}$, and fix $\m{G} = \m{A}^{1/2} \in \c{G}(\m{A})$ without loss of generality. Note that if $\m{G}$ and $\m{M}$ are properly aligned (as characterized in \cref{thm:Kanto:char}), then they satisfy the geometric mean identity \cite{bhatia2019bures}:
\begin{align*}
    (\m{G}^{*} \m{M})^{2} = \m{G}^{*} \m{B} \m{G} \quad \Longrightarrow \quad \m{G}^{*} \m{M} = (\m{A}^{1/2} \m{B} \m{A}^{1/2})^{1/2}.
\end{align*} 
In the finite-dimensional case $\c{H} = \b{R}^{d}$, various convergence results of the EOT coupling $\pi_{\varepsilon}$ to the unique unregularized OT coupling under certain regularity conditions or compactness of the support \cite{leonard2012schrodinger, carlier2017convergence, nutz2022entropic, carlier2023convergence, eckstein2022quantitative, ghosal2022stability}.
While the specific stability analysis for the Gaussian setting has been explored in \cite{janati2020entropic, goldfeld2020gaussian} (for $\b{R}^{d}$) and \cite{minh2023entropic} (for infinite dimensions), these works typically assume the injectivity of both covariance operators. 

However, the optimal Gaussian coupling is unique in the degenerate setting if and only if the Schur complement vanishes $\m{B}/\m{A} = \m{0}$ as stated in \cref{ssec:rev:OT}. When this condition fails, the set of optimal unregularized couplings is parametrized by the choice of $\m{N}_{{12}}  : \c{H}_{2} \to \c{H}_{1}$ in \cref{thm:Kanto:char} satisfying $\c{R}(\m{N}_{{12}}) \subset \c{N}(\m{B}_{{11}}^{1/2} \m{A}_{{11}}^{1/2})$ and $\m{N}_{{12}}^{*} \m{N}_{{12}} \preceq \m{B}/\m{A}$. Specifically, a coupling corresponds to a Monge map (i.e., is supported on a graph) if and only if the variance is minimized, which requires saturating the inequality $\m{N}_{{12}}^{*} \m{N}_{{12}} = \m{B}/\m{A}$.
In stark contrast, we claim that the entropic regularization mechanism selects the \emph{most diffusive} solution among the optimal set. This leads to a direct implication: \textbf{whenever the optimal transport solution is non-unique (i.e., $\m{B}/\m{A} \ne \m{0}$), the EOT limit is never a Monge solution}. Specifically, as $\varepsilon \to 0$, the EOT coupling $\m{\Sigma}_{\varepsilon} \in \c{B}_{1}^{+}(\c{H} \times \c{H})$ converges to the specific OT coupling defined by choosing $\m{N}_{{12}} = \m{0}$ (the minimal norm solution) in the Wasserstein distance.

In this regard, we will define the canonical choice of relevant operators to ease the notation. First, the properly aligned Green's operators are selected according to $\m{N}_{{12}} = \m{0}$:
\begin{align}\label{eq:canon:Green}
    \m{G}_{0} := \m{A}^{1/2} = \begin{pmatrix}
    \m{A}_{11}^{1/2} & \m{0} \\
    \m{0} & \m{0}
    \end{pmatrix} \in \c{G}(\m{A}), \quad \m{M}_{0} :=
    \begin{pmatrix}
    \m{A}_{11}^{\dagger/2} \m{P}_{11}^{1/2} & \m{0} \\
    [\m{P}_{11}^{\dagger/2} \m{A}_{{11}}^{1/2} \m{B}_{{12}}]^{*} & (\m{B}/\m{A})^{1/2}
    \end{pmatrix}  \in \c{G}(\m{B}),
\end{align}
where $\m{P}_{11} := \m{A}_{11}^{1/2} \m{B}_{11} \m{A}_{11}^{1/2} : \c{H}_{1} \to \c{H}_{1}$. Consequently, the proper alignement condition is:
\begin{align*}
    \m{G}^{*} \m{M} = \m{M}^{*} \m{G} = \m{P}^{1/2} \succeq \m{0}, \quad \m{P} :=  \m{A}^{1/2} \m{B} \m{A}^{1/2} = \begin{pmatrix}
    \m{P}_{11} & \m{0} \\
    \m{0} & \m{0}
    \end{pmatrix}. 
\end{align*}
Finally, the limit of the EOT coupling is given by
\begin{align*}
    \m{\Sigma}_{0} := 
    \left(
    \begin{array}{c|c}
    \m{A} & \m{C}_{0} \\
    \hline
    \m{C}_{0}^{*} & \m{B}
    \end{array}
    \right) \in \c{B}_{1}^{+}(\c{H} \times \c{H}).
\end{align*}   
where the cross-covariance is given by:
\begin{align}
    \m{C}_{0} := \m{G}_{0} \m{M}_{0}^{*} = \begin{pmatrix}
    \m{P}_{11}^{1/2}  & \m{B}_{11}^{1/2} \m{P}_{11}^{\dagger/2} \m{A}_{{11}}^{1/2} \m{B}_{{12}} \\
    \m{0} & \m{0}
    \end{pmatrix} \in \c{B}_{1}^{+}(\c{H}), \label{eq:limit:OT}
\end{align} 

\begin{example}[Orthogonal Marginals]
Consider the degenerate case on $\c{H} = \b{R}^{2}$ with orthogonal supports:
\begin{equation*}
    \m{A} = \begin{pmatrix}
    1 & 0 \\
    0 & 0
    \end{pmatrix}, \quad \m{B} = \begin{pmatrix}
    0 & 0 \\
    0 & 1
    \end{pmatrix}.
\end{equation*}
In this case, $\m{A} \m{B} = \m{0}$ and the standard OT problem admits infinite solutions. Any coupling of the form:
\begin{align*}
    \m{\Sigma} = \left(
    \begin{array}{cc|cc}
        1 & 0 & 0 & \rho \\
        0 & 0 & 0 & 0 \\
        \hline
        0 & 0 & 0 & 0 \\
        \rho & 0 & 0 & 1
    \end{array}
    \right), \quad |\rho| \le 1,
\end{align*}
is an optimal OT coupling, and the coupling corresponds to the Monge solution is obtained at the boundaries $\rho = \pm 1$ (perfect correlation or anti-correlation). On the other hand, \cref{cor:ortho:indep} shows that, for any $\varepsilon > 0$, the EOT solution is the independent coupling ($\rho = 0$):
\begin{align*}
    \m{\Sigma}_{\varepsilon} \equiv \m{\Sigma}_{0} = \left(
    \begin{array}{cc|cc}
        1 & 0 & 0 & 0 \\
        0 & 0 & 0 & 0 \\
        \hline
        0 & 0 & 0 & 0 \\
        0 & 0 & 0 & 1
    \end{array}
    \right),
\end{align*}
The regularization explicitly rejects the deterministic transport plans in favor of the maximum entropy (independent) plan.
\end{example}

We present another expression of $\m{M}_{0}$ that offers a computational advantage by bypassing the block decomposition in its implementation:
\begin{proposition}[One-Step Coordinate-Free Update]\label{prop:one:update}
Let $\m{A}, \m{B} \in \c{B}_{1}^{+}(\c{H})$ with $\m{A} \leadsto \m{B}$. For any $\m{M} \in \c{G}(\m{B})$, define
\begin{align*}
    \m{M}_{0} := \m{K}_{0} + (\m{B} - \m{K}_{0} \m{K}_{0}^{*})^{1/2}, \quad \m{K}_{0} = \m{M} (\m{A}^{1/2} \m{M})^{\dagger} \m{P}^{1/2}.
\end{align*}
Then, neither $\m{K}_{0}$ nor $\m{M}_{0}$ do not depend on the choice of $\m{M} \in \c{G}(\m{B})$. Furthermore, the second correction term vanishes if and only if $\m{B}/\m{A} = \m{0}$. In particular, $\m{M}_{0}$ coincides with the canonical choice in \eqref{eq:canon:Green}. 
\end{proposition}
\begin{proof}
Recall from \cref{ssec:rev:OT} that $\m{U}_{\m{M}} = \m{B}^{\dagger/2} \m{M}$ is the partial isometry with the final space $\c{R}(\m{U}_{\m{M}}) = \overline{\c{R}(\m{B})}$, thus $\m{U}_{\m{M}} \m{U}_{\m{M}}^{*} = \m{\Pi}_{\overline{\c{R}(\m{B})}}$. Consequently, 
\begin{align*}
    \m{M} (\m{A}^{1/2} \m{M})^{\dagger} \m{P}^{1/2} &= \m{B}^{1/2} \m{U}_{\m{M}} (\m{A}^{1/2} \m{B}^{1/2} \m{U})^{\dagger} \m{P}^{1/2} \\
    &= \m{B}^{1/2} \m{U}_{\m{M}} \m{U}_{\m{M}}^{*} (\m{A}^{1/2} \m{B}^{1/2})^{\dagger} \m{P}^{1/2} \\
    &= \m{B}^{1/2} \m{\Pi}_{\overline{\c{R}(\m{B})}} (\m{A}^{1/2} \m{B}^{1/2})^{\dagger} \m{P}^{1/2} = \m{B}^{1/2} (\m{A}^{1/2} \m{B}^{1/2})^{\dagger} \m{P}^{1/2},
\end{align*}
where the last equality holds since $\c{R}((\m{A}^{1/2}\m{B}^{1/2})^\dagger) = \c{R}(\m{B}^{1/2}\m{A}^{1/2}) \subseteq \overline{\c{R}(\m{B})}$. Thus, $\m{K}_{0}$ does not depend on the choice of $\m{M} \in \c{G}(\m{B})$. To derive the block representation of $\m{K}_{0}$, consider the specific Green's operator (which is simply $\m{M}_{0}$):
\begin{align*}
    \m{M} = \begin{pmatrix}
    \m{A}_{11}^{\dagger/2} \m{P}_{11}^{1/2} & \m{0} \\
    [\m{P}_{11}^{\dagger/2} \m{A}_{{11}}^{1/2} \m{B}_{{12}}]^{*} & (\m{B}/\m{A})^{1/2}
    \end{pmatrix}  \in \c{G}(\m{B}),
\end{align*}
which is a valid Green's operator due to \cite[Lemma 2.4.]{yun2025gaussian}:
\begin{equation}\label{eq:schur:invar}
        \m{B}/\m{A} = \m{B}_{22} - [\m{P}_{11}^{\dagger/2} \m{A}_{{11}}^{1/2} \m{B}_{{12}}]^{*} [\m{P}_{11}^{\dagger/2} \m{A}_{{11}}^{1/2} \m{B}_{{12}}],
\end{equation}
Therefore, we have $\m{A}^{1/2} \m{M} = \text{diag}(\m{P}_{11}^{1/2}, \m{0})$ and
\begin{align*}
    \m{K}_{0} = \begin{pmatrix}
    \m{A}_{11}^{\dagger/2} \m{P}_{11}^{1/2} & \m{0} \\
    [\m{P}_{11}^{\dagger/2} \m{A}_{{11}}^{1/2} \m{B}_{{12}}]^{*} & (\m{B}/\m{A})^{1/2}
    \end{pmatrix} \begin{pmatrix}
    \m{P}_{11}^{\dagger/2} & \m{0} \\
    \m{0} & \m{0}
    \end{pmatrix} \begin{pmatrix}
    \m{P}_{11}^{1/2} & \m{0} \\
    \m{0} & \m{0}
    \end{pmatrix} = \begin{pmatrix}
    \m{A}_{11}^{\dagger/2} \m{P}_{11}^{1/2} & \m{0} \\
    [\m{P}_{11}^{\dagger/2} \m{A}_{{11}}^{1/2} \m{B}_{{12}}]^{*} & \m{0}
    \end{pmatrix}.
\end{align*}
Consequently, using \eqref{eq:schur:invar}, we obtain:
\begin{align*}
    \m{B} - \m{K}_{0} \m{K}_{0}^{*} = \begin{pmatrix}
    \m{0} & \m{0} \\
    \m{0} & \m{B}/\m{A}
    \end{pmatrix}.
\end{align*}
Thus, $\m{M} = \m{K}_{0} + (\m{B} - \m{K}_{0} \m{K}_{0}^{*})^{1/2}$ recovers the block form in \eqref{eq:canon:Green}.
\end{proof}

\begin{remark}[Construction of Proper Alignment]\label{rmk:one:update}
Notably, \cref{prop:one:update} generalizes to a constructive method for obtaining a properly aligned Green's operator from arbitrary initial choices. Given $\m{A}, \m{B} \in \c{B}_{1}^{+}(\c{H})$, let $\m{G} \in \c{G}(\m{A})$ and $\m{M} \in \c{G}(\m{B})$ be any Green's operators (e.g., obtained via Cholesky decomposition). We can update $\m{M}$ to satisfy the alignment condition with $\m{G}$ by defining:
\begin{align*}
    \m{K} = \m{M} (\m{G}^{*} \m{M})^{\dagger} (\m{G}^{*} \m{B} \m{G})^{1/2} \quad \text{and} \quad \m{L} \in \c{G}(\m{B} - \m{K} \m{K}^{*}).
\end{align*}
Then, setting $\widetilde{\m{M}} := \m{K} + \m{L}$ guarantees that $\widetilde{\m{M}} \in \c{G}(\m{B})$ and $\m{G}^{*} \widetilde{\m{M}} = (\m{G}^{*} \m{B} \m{G})^{1/2} \succeq \m{0}$. This construction admits an interpretation as an operator-valued Pythagorean theorem with respect to the Bures-Wasserstein geometry:
\begin{align*}
    \underbrace{\m{M} \m{M}^{*}}_{\substack{\text{total covariance} \\ (= \m{B})}} = \underbrace{\m{K} \m{K}^{*}}_{\substack{\text{aligned covariance} \\ (\text{projection onto } \c{R}(\m{M}^{*}\m{G}))}} + \underbrace{\m{L} \m{L}^{*}}_{\substack{\text{residual variance} \\ (= \text{diag}(\m{0}, \m{B}/\m{A}))}}.
\end{align*}
This intrinsic geometric structure remains obscured in the dominant literature, which typically relies on the assumption $\c{N}(\m{A}) \subseteq \c{N}(\m{B})$ or non-singular marginals. In such settings, the Schur complement vanishes, leading standard approaches to default to symmetric square roots $\m{G}=\m{A}^{1/2}$ and $\m{M}=\m{B}^{1/2}$—which still fail to be properly aligned in the non-commutative case.
\end{remark}

\begin{theorem}\label{thm:EOT:conv}
Let $\m{A}, \m{B} \in \c{B}_{1}^{+}(\c{H})$ with $\m{A} \leadsto \m{B}$. Then, the cross-covariance
$\m{C}_{\varepsilon} \in \c{B}_{1}(\c{H})$ in \eqref{eq:opt:cross-cov} converges to $\m{C}_{0} \in \c{B}_{1}(\c{H})$ in \eqref{eq:limit:OT} in the trace-class norm $\vertii{\cdot}_1$. Consequently, the Gaussian EOT coupling $\pi_{\varepsilon} = \c{N}(\m{0}, \m{\Sigma}_{\varepsilon})$ converges to the canonical OT coupling $\pi_{0} = \c{N}(\m{0}, \m{\Sigma}_{0})$ in the 2-Wasserstein distance:
\begin{align*}
    \lim_{\varepsilon \to 0} \c{W}_{2}(\pi_{\varepsilon}, \pi_{0}) = 0.
\end{align*}
\end{theorem}
\begin{proof}
The optimal entropic cross-covariance is given by $\m{C}_{\varepsilon} = \m{G}_{0} f_{\varepsilon}(\m{G}_{0}^{*} \m{M}_{0}) \m{M}_{0}^*$. 
Consider the limit operator $\m{R}_{0} := f_{0}(\m{G}_{0}^{*} \m{M}_{0}) = \text{diag}(\m{\Pi}_{\overline{\c{R}(\m{P}_{11})}}, \m{0})$. As $\varepsilon \to 0$, $f_{\varepsilon}(\m{G}_{0}^{*} \m{M}_{0}) \xrightarrow{SOT} \m{R}_{0}$ by spectral calculus properties. Applying \cref{lem:Holder:SOT}, we obtain convergence in trace norm:
\begin{align*}
    \m{C}_{\varepsilon} \xrightarrow{\vertii{\cdot}_{1}} \m{G}_{0} \begin{pmatrix} \m{\Pi}_{\overline{\c{R}(\m{P}_{11})}} & \m{0} \\ \m{0} & \m{0} \end{pmatrix} \m{M}_{0}^* = \m{G}_{0} \begin{pmatrix} [\m{A}_{{11}}^{\dagger/2} \m{P}_{11}^{1/2}]^{*} & \m{P}_{11}^{\dagger/2} \m{A}_{{11}}^{1/2} \m{B}_{{12}} \\ \m{0} & \m{0} \end{pmatrix} = \m{G}_{0} \m{M}_{0}^{*} = \m{C}_{0}.
\end{align*}

The squared Bures-Wasserstein distance is given by:
\begin{align*}
    \c{W}_{2}^{2}(\pi_{\varepsilon}, \pi_{0}) = \tr(\m{\Sigma}_{\varepsilon}) + \tr(\m{\Sigma}_{0}) - 2 \tr [(\m{\Sigma}_{\varepsilon}^{1/2} \m{\Sigma}_{0} \m{\Sigma}_{\varepsilon}^{1/2})^{1/2}].
\end{align*}
The convergence $\vertii{\m{\Sigma}_{\varepsilon} - \m{\Sigma}_{0}}_{1} \to 0$ implies convergence of the traces $\tr(\m{\Sigma}_{\varepsilon}) \to \tr(\m{\Sigma}_{0})$. Furthermore, the map $\m{X} \in \c{B}_1^+(\c{H}) \mapsto \tr(\sqrt{\m{X}}) \in [0, \infty)$ is continuous. Since $\m{\Sigma}_{\varepsilon}^{1/2} \m{\Sigma}_{0} \m{\Sigma}_{\varepsilon}^{1/2} \to \m{\Sigma}_{0}^2$ in $\c{B}_{1}^{+}(\c{H})$, the cross-term converges to $\tr(\sqrt{\m{\Sigma}_0^2}) = \tr(\m{\Sigma}_0)$. Thus, the distance vanishes.
\end{proof}

To establish the asymptotic convergence rate, we first derive an explicit expression for the squared Wasserstein distance. To the best of our knowledge, this is the first closed-form formulation expressed exclusively via operators on the marginal space $\c{H}$, bypassing the need to operate on $\c{H} \times \c{H}$.

\begin{theorem}\label{thm:W2:exact}
Let $\varepsilon > 0$ and $\m{A}, \m{B} \in \c{B}_{1}^{+}(\c{H})$ with $\m{A} \leadsto \m{B}$. Then, the squared Wasserstein distance between the Gaussian EOT coupling $\pi_{\varepsilon} = \c{N}(\m{0}, \m{\Sigma}_{\varepsilon})$ and the canonical OT coupling $\pi_{0} = \c{N}(\m{0}, \m{\Sigma}_{0})$ is given by
\begin{align*}
    \c{W}_{2}^{2}(\pi_{\varepsilon}, \pi_{0}) = 2 \tr [\m{A} + \m{B}- \sqrt{\m{Q}_{\varepsilon}}],
\end{align*}
where
\begin{align*}
    \m{Q}_{\varepsilon} :&= (\m{G}_{0}^{*} \m{G}_{0})^{2} + (\m{M}_{0}^{*} \m{M}_{0})^{2} + (\m{G}_{0}^{*} \m{G}_{0}) \m{R}_{\varepsilon} (\m{M}_{0}^{*} \m{M}_{0}) + (\m{M}_{0}^{*} \m{M}_{0}) \m{R}_{\varepsilon} (\m{G}_{0}^{*} \m{G}_{0}). 
    % \\
    % &= \begin{pmatrix}
    % \m{G}_{0}^{*} \m{G}_{0} & \m{M}_{0}^{*} \m{M}_{0}
    % \end{pmatrix} \begin{pmatrix} \m{I} & \m{R}_{\varepsilon} \\ \m{R}_{\varepsilon} & \m{I} \end{pmatrix} \begin{pmatrix}
    % \m{G}_{0}^{*} \m{G}_{0} \\ \m{M}_{0}^{*} \m{M}_{0}
    % \end{pmatrix} \in \c{B}_{0}^{+} (\c{H}).
\end{align*}
\end{theorem}
\begin{proof}
Noting that
\begin{align*}
    \m{\Sigma}_{0} &= \left(
    \begin{array}{c|c}
    \m{A} & \m{G}_{0} \m{M}_{0}^{*} \\
    \hline
    \m{M}_{0} \m{G}_{0}^{*} & \m{B}
    \end{array}
    \right) = \left(
    \begin{array}{c}
    \m{G}_{0} \\
    \hline
    \m{M}_{0}
    \end{array}
    \right) \left(
    \begin{array}{c|c}
    \m{G}_{0}^{*} & \m{M}_{0}^{*}
    \end{array}
    \right),
\end{align*}
the spectrums of $\m{\Sigma}_{\varepsilon}^{1/2} \m{\Sigma}_{0} \m{\Sigma}_{\varepsilon}^{1/2}$ and
\begin{align*}
    \m{Q}_{\varepsilon} &= \left[\m{\Sigma}_{\varepsilon}^{1/2} \left(
    \begin{array}{c}
    \m{G}_{0} \\
    \hline
    \m{M}_{0}
    \end{array}
    \right) \right]^{*} \left[\m{\Sigma}_{\varepsilon}^{1/2} \left(
    \begin{array}{c}
    \m{G}_{0} \\
    \hline
    \m{M}_{0}
    \end{array}
    \right) \right] = \left(
    \begin{array}{c|c}
    \m{G}_{0}^{*} & \m{M}_{0}^{*}
    \end{array}
    \right) \m{\Sigma}_{\varepsilon} \left(
    \begin{array}{c}
    \m{G}_{0} \\
    \hline
    \m{M}_{0}
    \end{array}
    \right) \\
    &= (\m{G}_{0}^{*} \m{G}_{0})^{2} + (\m{M}_{0}^{*} \m{M}_{0})^{2} + (\m{G}_{0}^{*} \m{G}_{0}) \m{R}_{\varepsilon} (\m{M}_{0}^{*} \m{M}_{0}) + (\m{M}^{*} \m{M}) \m{R}_{\varepsilon} (\m{G}_{0}^{*} \m{G}_{0}),
\end{align*}
coincide. Therefore, we conclude that
\begin{align*}
    \c{W}_{2}^{2}(\pi_{\varepsilon}, \pi_{0}) = \tr[\m{\Sigma}_{\varepsilon}] + \tr[\m{\Sigma}_{0}] - 2 \tr [\m{\Sigma}_{\varepsilon}^{1/2} \m{\Sigma}_{0} \m{\Sigma}_{\varepsilon}^{1/2}] = 2 \tr [\m{A} + \m{B}- \sqrt{\m{Q}_{\varepsilon}}].
\end{align*}
\end{proof}

To investigate the asymptotic rate, we begin by the benchmark case $\m{A} = \m{B}$, where the simplification follows immediately by considering $\m{G}_0 = \m{M}_0 = \m{A}^{1/2}$ in \cref{thm:W2:exact}: $\c{W}_{2}^{2}(\pi_{\varepsilon}, \pi_{0}) = 2 \tr [\m{A} \delta_{\varepsilon}(\m{A})]$.
Here, the \textbf{perturbation function} $\delta_{\varepsilon} : [0, \infty) \to [0, 2-\sqrt{2})$ is defined as:
\begin{align}\label{eq:def:pert:ftn}
    \delta_{\varepsilon}(x) = \begin{cases}
        2 -  \sqrt{2(1+f_{\varepsilon}(x))} &, \quad x>0, \\
        0 &, \quad x=0,
    \end{cases}
\end{align}
as depicted in \cref{fig:three_plots}. Hence, if $\m{A}=\m{B}$ and they have a finite rank, then one can show:
\begin{align}\label{eq:univ:bdd:same}
    \lim_{\varepsilon \downarrow 0} \frac{x \delta_{\varepsilon}(x)}{\varepsilon} = \frac{1}{4} \quad \implies \quad \lim_{\varepsilon \downarrow 0} \frac{\c{W}_{2}^{2}(\pi_{\varepsilon}, \pi_{0})}{\varepsilon} = \frac{\mathrm{rank}(\m{A})}{2},
\end{align}
using the scaling property of the perturbation function $\delta_{\varepsilon}$ in \cref{lem:pert:ftn:mono}. However, in the infinite-dimensional setup, the nonparametric rate even of the simpliest scenario in the following corollary alludes that the exact asymptotic rate is intractable without further structural assumptions. 

\begin{corollary}[Nonparametric Rate]
Let $\varepsilon > 0$ and $\m{A} = \m{B} \in \c{B}_{1}^{+}(\c{H})$. If the eigenvalues of $\m{A}$ decay polynomially as $\lambda_{k} = c \, k^{-\alpha}$ for some $c > 0$ and $\alpha > 1$, then the exact asymptotic convergence rate is:
\begin{align*}
    \lim_{\varepsilon \downarrow 0} \frac{\c{W}_{2}^{2}(\pi_{\varepsilon}, \pi_{0})}{\varepsilon^{1-1/\alpha}} = 2 c^{1/\alpha} \int_{0}^{\infty} t^{-\alpha} \delta_{1}(t^{-\alpha}) \, \rd t \in (0, \infty).
\end{align*}
\end{corollary}
\begin{proof}
Recall the scaling property from \cref{lem:pert:ftn:mono}:
\begin{align*}
    \frac{x \delta_{\varepsilon}(x)}{\varepsilon} = \psi \left( \frac{\varepsilon}{2x} \right) = \frac{(x/\varepsilon) \delta_{1}(x/\varepsilon)}{1} \quad \implies \delta_{\varepsilon}(x) = \delta_1(x/\varepsilon).
\end{align*}
We express the trace as a spectral sum by substituting $\lambda_k = c k^{-\alpha}$:
\begin{align*}
    \c{W}_{2}(\pi_{\varepsilon}, \pi_{0}) = 2 \sum_{k=1}^{\infty} \lambda_k \delta_{\varepsilon}(\lambda_k) = 2 \sum_{k=1}^{\infty} c k^{-\alpha} \delta_{1}\left( \frac{c k^{-\alpha}}{\varepsilon} \right).
\end{align*}
We identify this sum as a Riemann approximation of an integral. Let $u_k := k (\varepsilon/c)^{1/\alpha}$. Then $k = u_k (c/\varepsilon)^{1/\alpha}$ and the mesh size is $\Delta u = (\varepsilon/c)^{1/\alpha}$. The term inside the sum becomes:
\begin{align*}
    c k^{-\alpha} \delta_{1}\left( \frac{c k^{-\alpha}}{\varepsilon} \right) = c \left( u_k (c/\varepsilon)^{1/\alpha} \right)^{-\alpha} \delta_1(u_k^{-\alpha}) = \varepsilon u_k^{-\alpha} \delta_1(u_k^{-\alpha}).
\end{align*}
Substituting this back into the sum:
\begin{align*}
    \c{W}_{2}(\pi_{\varepsilon}, \pi_{0}) = \sum_{k=1}^{\infty} \varepsilon u_k^{-\alpha} \delta_1(u_k^{-\alpha}) 
    = \varepsilon (\varepsilon/c)^{-1/\alpha} \sum_{k=1}^{\infty} u_k^{-\alpha} \delta_1(u_k^{-\alpha}) \Delta u 
    = c^{1/\alpha} \varepsilon^{1-1/\alpha} \left[ \sum_{k=1}^{\infty} u_k^{-\alpha} \delta_1(u_k^{-\alpha}) \Delta u \right].
\end{align*}
Since the function $h(t) = t^{-\alpha} \delta_1(t^{-\alpha})$ is Riemann integrable on $(0, \infty)$ (decaying as $O(t^{-\alpha})$ at infinity and bounded by $O(t^\alpha \cdot t^{-\alpha}) \sim O(1)$ near zero due to \cref{lem:pert:ftn:mono}), the sum converges to the integral $\int_0^\infty h(t) dt$ as $\Delta u \to 0$ (i.e., $\varepsilon \to 0$). Thus, the limit exists and is strictly positive.
\end{proof}

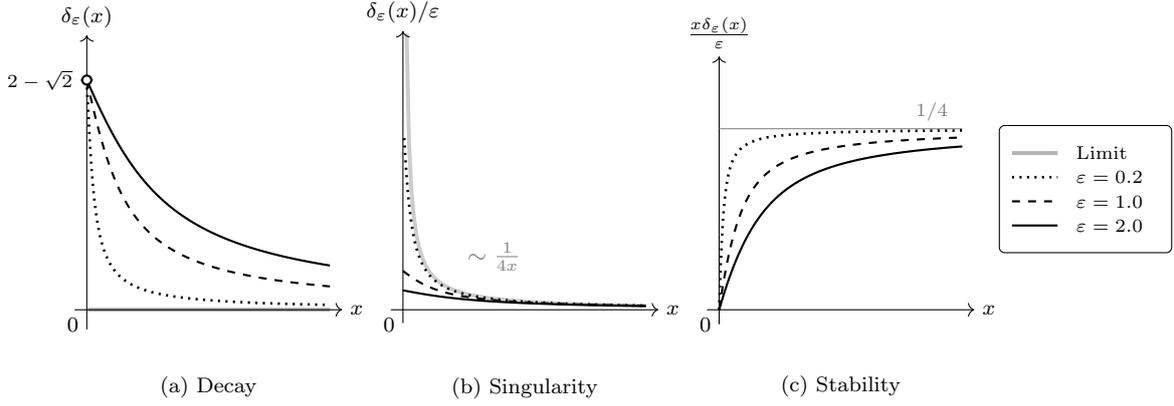
\begin{figure}[h!]
    \centering
    \begin{tikzpicture}[scale=0.8] % Scaled down slightly to fit 3 plots
        % -----------------------------------------------------------------
        % FUNCTIONS
        % f(x, eps) = 2x / (sqrt(4x^2 + eps^2) + eps)
        \def\funcF#1#2{ (2*(#1) / (sqrt(4*(#1)*(#1) + (#2)*(#2)) + (#2))) }
        
        % delta(x, eps) = 2 - sqrt(2*(1 + f))
        \def\funcD#1#2{ (2 - sqrt(2 * (1 + \funcF{#1}{#2}))) }

        % -----------------------------------------------------------------
        % LEFT PLOT: delta_epsilon(x)
        % Y-Range: [0, 0.6]. Visual Height ~4 units. yscale = 6.5
        % -----------------------------------------------------------------
        \begin{scope}[local bounding box=LeftPlot, yscale=6.5]
            % Axes
            \draw[->] (-0.2, 0) -- (4.2, 0) node[right] {$x$};
            \draw[->] (0, -0.05) -- (0, 0.7) node[above] {$\delta_{\varepsilon}(x)$};
            \node at (0,0) [below left] {$0$};

            % Y-Tick for 2 - sqrt(2) approx 0.5857
            \draw (-0.1, 0.5857) -- (0.1, 0.5857);
            \node[left, align=right, font=\scriptsize] at (-0.1, 0.5857) {$2-\sqrt{2}$};

            % Curves
            % Limit case (Epsilon -> 0)
            \draw[line width=1.5pt, black, opacity=0.3] (0.02, 0) -- (4.0, 0); 

            % Epsilon = 2.0
            \draw[thick, black, domain=0:4.0, samples=100] plot (\x, {\funcD{\x}{2.0}});
            % Epsilon = 1.0
            \draw[thick, black, dashed, domain=0:4.0, samples=100] plot (\x, {\funcD{\x}{1.0}});
            % Epsilon = 0.2
            \draw[thick, black, dotted, line width=1pt, domain=0:4.0, samples=200] plot (\x, {\funcD{\x}{0.2}});

            % Singularity Node (Circle shape preserved)
            \node[circle, fill=white, draw=black, thick, inner sep=1.2pt] at (0, 0.5857) {};

            % Label
            \node[anchor=north] at (2, -0.15) {(a) Decay};
        \end{scope}

        % -----------------------------------------------------------------
        % MIDDLE PLOT: delta_epsilon(x) / epsilon
        % Y-Range: [0, 4.0]. Visual Height ~4 units. yscale = 1.1
        % -----------------------------------------------------------------
        \begin{scope}[xshift=5.2cm, local bounding box=MidPlot, yscale=1.1]
            % Axes
            \draw[->] (-0.2, 0) -- (4.2, 0) node[right] {$x$};
            \draw[->] (0, -0.2) -- (0, 4.2) node[above] {$\delta_{\varepsilon}(x)/\varepsilon$};
            \node at (0,0) [below left] {$0$};

            % Curves
            % Limit Envelope 1/(4x)
            \draw[line width=1.5pt, black, opacity=0.2, domain=0.06:4.0, samples=100] 
                plot (\x, { 1 / (4*\x) });
            \node[gray, font=\small] at (1.5, 0.8) {$\sim \frac{1}{4x}$};

            % Epsilon = 2.0
            \draw[thick, black, domain=0:4.0, samples=100] plot (\x, { \funcD{\x}{2.0} / 2.0 });
            % Epsilon = 1.0
            \draw[thick, black, dashed, domain=0:4.0, samples=100] plot (\x, { \funcD{\x}{1.0} / 1.0 });
            % Epsilon = 0.2
            \draw[thick, black, dotted, line width=1pt, domain=0.02:4.0, samples=200] plot (\x, { \funcD{\x}{0.2} / 0.2 });

            % Label
            \node[anchor=north] at (2, -0.9) {(b) Singularity};
        \end{scope}

        % -----------------------------------------------------------------
        % RIGHT PLOT: x * delta_epsilon(x) / epsilon
        % Y-Range: [0, 0.35]. Visual Height ~4 units. yscale = 12
        % Max value tends to 0.25. 0.35 * 12 = 4.2 (height matches)
        % -----------------------------------------------------------------
        \begin{scope}[xshift=10.4cm, local bounding box=RightPlot, yscale=12]
            % Axes
            \draw[->] (-0.2, 0) -- (4.2, 0) node[right] {$x$};
            \draw[->] (0, -0.02) -- (0, 0.35) node[above] {$\frac{x \delta_{\varepsilon}(x)}{\varepsilon}$};
            \node at (0,0) [below left] {$0$};

            % Limit Line y = 1/4
            \draw[gray, thin] (0, 0.25) -- (4.0, 0.25);
            \node[gray, font=\scriptsize, above] at (3.5, 0.25) {$1/4$};

            % Curves
            % Epsilon = 2.0
            \draw[thick, black, domain=0:4.0, samples=100] plot (\x, { \x * \funcD{\x}{2.0} / 2.0 });
            % Epsilon = 1.0
            \draw[thick, black, dashed, domain=0:4.0, samples=100] plot (\x, { \x * \funcD{\x}{1.0} / 1.0 });
            % Epsilon = 0.2
            \draw[thick, black, dotted, line width=1pt, domain=0:4.0, samples=200] plot (\x, { \x * \funcD{\x}{0.2} / 0.2 });

            % Label
            \node[anchor=north] at (2, -0.08) {(c) Stability};
        \end{scope}

        % -----------------------------------------------------------------
        % LEGEND BOX
        % -----------------------------------------------------------------
        % Placed to the right of the third plot
        \node[draw, fill=white, align=left, anchor=west, inner sep=4pt, rounded corners=2pt] 
            at (15.0, 2.0) {
            \begin{tikzpicture}[baseline=(current bounding box.center), yscale=0.8]
                % Legend Entry 1 (Top)
                \draw[line width=1.5pt, black, opacity=0.3] (0, 1.2) -- (0.6, 1.2);
                \node[anchor=west, font=\scriptsize] at (0.7, 1.2) {Limit};

                % Legend Entry 2
                \draw[thick, black, dotted, line width=1pt] (0, 0.8) -- (0.6, 0.8);
                \node[anchor=west, font=\scriptsize] at (0.7, 0.8) {$\varepsilon=0.2$};

                % Legend Entry 3
                \draw[thick, black, dashed] (0, 0.4) -- (0.6, 0.4);
                \node[anchor=west, font=\scriptsize] at (0.7, 0.4) {$\varepsilon=1.0$};

                % Legend Entry 4 (Bottom)
                \draw[thick, black] (0, 0) -- (0.6, 0);
                \node[anchor=west, font=\scriptsize] at (0.7, 0) {$\varepsilon=2.0$};
            \end{tikzpicture}
        };

    \end{tikzpicture}
    \caption{Behavior of the perturbation function. (a) $\delta_{\varepsilon}(x) = \delta_1(x/\varepsilon)$ decays rapidly but starts at a fixed singularity at $x=0$. (b) $\delta_{\varepsilon}(x)/\varepsilon$ reveals a $1/x$ singularity near the origin, which suggests a nonparametric rate. (c) Weighting by $x$ regularizes the singularity, bounded by $1/4$.}
    \label{fig:three_plots}
\end{figure}

We present two upper bounds of the square Wasserstein distance:
\begin{theorem}[Upper Bounds]\label{thm:upper:bdds}
Let $\varepsilon > 0$ and $\m{A}, \m{B} \in \c{B}_{1}^{+}(\c{H})$ with $\m{A} \leadsto \m{B}$. Then, the squared Wasserstein distance between the Gaussian EOT coupling $\pi_{\varepsilon} = \c{N}(\m{0}, \m{\Sigma}_{\varepsilon})$ and the canonical OT coupling $\pi_{0} = \c{N}(\m{0}, \m{\Sigma}_{0})$ can be bounded by:
\begin{align}
    &\c{W}_{2}^{2}(\pi_{\varepsilon}, \pi_{0}) \le \tr \left[(\m{G}_{0}^{*} \m{G}_{0} + \m{M}_{0}^{*} \m{M}_{0}) \delta_{\varepsilon}(\m{P}^{1/2}) \right], \label{eq:W2:pert:bdd1} \\
    &\c{W}_{2}^{2}(\pi_{\varepsilon}, \pi_{0}) \le 2 \min \left[ \tr \left((\m{G}_{0}^{*} \m{G}_{0}) \eta_{\varepsilon}(\m{P}^{1/2}) \right), \tr \left((\m{M}_{0}^{*} \m{M}_{0}) \eta_{\varepsilon}(\m{P}^{1/2}) \right) \right]. \label{eq:W2:pert:bdd2}
\end{align}
where $\eta_{\varepsilon} : [0, \infty) \to [0, 1)$ is defined as:
\begin{align*}
    \eta_{\varepsilon}(x) = \begin{cases}
        1-f_{\varepsilon}(x) &, \quad x>0, \\
        0 &, \quad x=0,
    \end{cases}
\end{align*}
As $\varepsilon \to 0$, we have:
\begin{align*}
    &\lim_{\varepsilon \downarrow 0} \tr \left((\m{G}_{0}^{*} \m{G}_{0} + \m{M}_{0}^{*} \m{M}_{0}) \delta_{\varepsilon}(\m{P}^{1/2}) \right) = (2-\sqrt{2}) \tr (\m{M}_{0} \m{\Pi}_{\c{N}(\m{P})} \m{M}_{0}^{*}), \\
    &\lim_{\varepsilon \downarrow 0} \tr \left[ \m{G}_{0}^{*} \m{G}_{0} \eta_{\varepsilon} (\m{P}^{1/2}) \right] = \lim_{\varepsilon \downarrow 0} \tr \left[ \m{M}_{0}^{*} \m{M}_{0} \eta_{\varepsilon} (\m{P}^{1/2}) \right] = 0.
\end{align*}
Note that $\tr (\m{M}_{0} \m{\Pi}_{\c{N}(\m{P})} \m{M}_{0}^{*}) = 0$ if and only if the Schur complement vanishes, i.e. $\m{B}/\m{A} = \m{0}$.
\end{theorem}
\begin{proof}
Recall from \eqref{eq:W2:bdd:HS} that the Wasserstein distance can be bounded from above by the Hilbert-Schmidt distance between Green's operators. We then consider two different factorizations. Recall from \cref{thm:EOT:conv} that $\m{R}_{\varepsilon} = f_{\varepsilon}(\m{G}_{0}^{*} \m{M}_{0}) \xrightarrow{SOT} \m{R}_{0} = f_{0}(\m{G}_{0}^{*} \m{M}_{0}) = \m{\Pi}_{\overline{\c{R}(\m{P})}}$.

\begin{enumerate}[leftmargin = *]
\item \textbf{Spectral Decomposition}: Let $\m{S} = \text{diag}(\m{G}_{0}, \m{M}_{0})$ to ease the notation, and introduce the Hadamard operator, analogous to the Hadamard gate in quantum mechanics \cite{hall2013quantum}:
\begin{align*}
    \m{H} = \frac{1}{\sqrt{2}} \left(
    \begin{array}{c|c}
    \m{I} & \m{I} \\
    \hline
    \m{I} & - \m{I}
    \end{array}
    \right) \in \c{B}_{\infty}(\c{H} \times \c{H}).
\end{align*}
Note that $\m{H}$ is unitary and symmetric, with eigenvalues $\pm 1$. Let $\m{P} := \m{A}^{1/2} \m{B} \m{A}^{1/2}$. Under the choice $\m{N}_{12}=\m{0}$, we have $\m{G}^*\m{M} = \m{P}^{1/2}$. The spectral shrinkage converges strongly, 
The covariances factorize as:
\begin{align*}
    \m{\Sigma}_{\varepsilon} = \m{S} \left(
    \begin{array}{c|c}
    \m{I} & \m{R}_{\varepsilon} \\
    \hline
    \m{R}_{\varepsilon} & \m{I}
    \end{array}
    \right) \m{S}^* = (\m{S} \m{H}) \m{D}_{\varepsilon} (\m{S} \m{H})^{*}, \quad \m{S} = \text{diag}(\m{G}, \m{M}), \quad \m{D}_{\varepsilon} = \text{diag}(\m{I} + \m{R}_{\varepsilon}, \m{I} - \m{R}_{\varepsilon}),
\end{align*}
and similarly:
\begin{align*}
    \m{\Sigma}_{0} = (\m{S} \m{H}) \m{D}_{0} (\m{S} \m{H})^{*}, \quad \m{D}_{0} = \text{diag}(\m{I} + \m{R}_{0}, \m{I} - \m{R}_{0}).
\end{align*}

Since $\m{S} \m{H} \m{D}_{\varepsilon}^{1/2} \in \c{G}(\m{\Sigma}_{\varepsilon})$ for any $\varepsilon \ge 0$, we get from \eqref{eq:W2:bdd:HS} that
\begin{align*}
    \c{W}_{2}^{2}(\m{\Sigma}_{\varepsilon}, \m{\Sigma}_{0}) \le \vertii{\m{S} \m{H} (\m{D}_{\varepsilon}^{1/2} -  \m{D}_{0}^{1/2})}_{2}^{2} = \tr [(\m{D}_{\varepsilon}^{1/2} -  \m{D}_{0}^{1/2})^{2} (\m{S} \m{H})^{*} (\m{S} \m{H})].
\end{align*}
It is the direct computation that
\begin{align*}
    (\m{S} \m{H})^{*} (\m{S} \m{H}) = \frac{1}{2} \left(
    \begin{array}{c|c}
        (\m{G}_{0}^{*} \m{G}_{0} + \m{M}_{0}^{*} \m{M}_{0}) & (\m{G}_{0}^{*} \m{G}_{0} - \m{M}_{0}^{*} \m{M}_{0}) \\
        \hline
        (\m{G}_{0}^{*} \m{G}_{0} - \m{M}_{0}^{*} \m{M}_{0}) & (\m{G}_{0}^{*} \m{G}_{0} + \m{M}_{0}^{*} \m{M}_{0})
    \end{array}
    \right).
\end{align*}
The operator $(\m{D}_{\varepsilon}^{1/2} - \m{D}_{0}^{1/2})^{2}$ is block diagonal. Multiplying and taking the trace, the off-diagonal terms vanish. We remain with:
\begin{align*}
    &\tr [(\m{D}_{\varepsilon}^{1/2} -  \m{D}_{0}^{1/2})^{2} (\m{S} \m{H})^{*} (\m{S} \m{H})] \\
    &= \frac{1}{2} \tr \left((\m{G}_{0}^{*} \m{G}_{0} + \m{M}_{0}^{*} \m{M}_{0}) [(\sqrt{\m{I} + \m{R}_{\varepsilon}} - \sqrt{\m{I} + \m{R}_{0}})^{2} + (\sqrt{\m{I} - \m{R}_{\varepsilon}} - \sqrt{\m{I} - \m{R}_{0}})^{2}] \right) \\
    &= \tr \left((\m{G}_{0}^{*} \m{G}_{0} + \m{M}_{0}^{*} \m{M}_{0}) \delta_{\varepsilon}(\m{P}^{1/2}) \right) ,
\end{align*}
In the last line, we used the fact that
\begin{align*}
    \delta_{\varepsilon}(x) = \frac{1}{2} \left[ \sqrt{1+f_{\varepsilon}(x)} - \sqrt{1+f_{0}(x)}]^{2} + [\sqrt{1-f_{\varepsilon}(x)} - \sqrt{1-f_{0}(x)} \right]^{2}, \quad x \in [0, \infty), \quad \varepsilon >0.
\end{align*}
Finally, we show that the upper bound converges to $0$ as $\varepsilon \to 0$ if and only if $\m{B}/\m{A} = \m{0}$. To this end, observe that $\delta_{\varepsilon}(\m{P}^{1/2}) \xrightarrow{SOT} (2-\sqrt{2}) \m{\Pi}_{\c{N}(\m{P})}$ due to \cref{lem:pert:ftn:mono}, thus \cref{lem:Holder:SOT} yields:
\begin{align*}
    \lim_{\varepsilon \downarrow 0} \tr \left((\m{G}_{0}^{*} \m{G}_{0} + \m{M}_{0}^{*} \m{M}_{0}) \delta_{\varepsilon}(\m{P}^{1/2}) \right) &= (2-\sqrt{2}) \tr \left((\m{G}_{0}^{*} \m{G}_{0} + \m{M}_{0}^{*} \m{M}_{0}) \m{\Pi}_{\c{N}(\m{P})} \right) \\
    &= (2-\sqrt{2}) \tr (\m{M}_{0} \m{\Pi}_{\c{N}(\m{P})} \m{M}_{0}^{*}).
\end{align*}
Note that this term vanishes if and only if $\m{M}_{0} \m{\Pi}_{\c{N}(\m{P})} = \m{0}$, or equivalently, $\c{N}(\m{P}) \subset \c{N}(\m{M}_{0})$. Noting that
\begin{align*}
    \left(
    \begin{array}{c}
        h \\
        \hline
        g 
    \end{array}
    \right) \in \c{N}(\m{P}) = \left\{ \left(
    \begin{array}{c}
        h \\
        \hline
        g 
    \end{array}
    \right) : h \in \c{N}(\m{P}_{11}), g \in \c{H} \right\} \quad \implies \quad \m{M}_{0} \left(
    \begin{array}{c}
        h \\
        \hline
        g 
    \end{array}
    \right) =
    \left(
    \begin{array}{c}
        \m{0} \\
        \hline
        (\m{B}/\m{A})^{1/2} g 
    \end{array}
    \right),
\end{align*} 
the upper bound converges to $0$ if and only if $\m{B}/\m{A} = \m{0}$.

\item \textbf{Cholesky Decomposition}: Observe that
\begin{align*}
    \left(
    \begin{array}{c|c}
    \m{I} & \m{R}_{\varepsilon} \\
    \hline
    \m{R}_{\varepsilon} & \m{I}
    \end{array}
    \right) = \left(
    \begin{array}{c|c}
    \m{I} & \m{0} \\
    \hline
    \m{R}_{\varepsilon} & \m{\Delta}_{\varepsilon}
    \end{array}
    \right) \left(
    \begin{array}{c|c}
    \m{I} & \m{R}_{\varepsilon} \\
    \hline
    \m{0} & \m{\Delta}_{\varepsilon}
    \end{array}
    \right), \quad \m{\Delta}_{\varepsilon} := (\m{I} - \m{R}_{\varepsilon}^{2})^{1/2},
\end{align*}
hence
\begin{align*}
    \m{S} \left(
    \begin{array}{c|c}
    \m{I} & \m{0} \\
    \hline
    \m{R}_{\varepsilon} & \m{\Delta}_{\varepsilon}
    \end{array}
    \right) = \left(
    \begin{array}{c|c}
    \m{G}_{0} & \m{0} \\
    \hline
    \m{M}_{0} \m{R}_{\varepsilon} & \m{M}_{0} \m{\Delta}_{\varepsilon}
    \end{array}
    \right) \in \c{G}(\m{\Sigma}_{\varepsilon}), \quad \varepsilon \ge 0.
\end{align*}
Thus, we get from \eqref{eq:W2:bdd:HS} that
\begin{align*}
    \c{W}_{2}^{2}(\m{\Sigma}_{\varepsilon}, \m{\Sigma}_{0}) \le \vertiii{\left(
    \begin{array}{c|c}
    \m{0} & \m{0} \\
    \hline
    \m{M}_{0} (\m{R}_{\varepsilon} - \m{R}_{0}) & \m{M}_{0} (\m{\Delta}_{\varepsilon} - \m{\Delta}_{0})
    \end{array}
    \right)}_{2}^{2} = 2 \tr \left[ \m{M}_{0}^{*} \m{M}_{0} \eta_{\varepsilon} (\m{P}^{1/2}) \right].
\end{align*}
For the equality, we used the fact that
\begin{align*}
    2 \eta_{\varepsilon}(x) = (f_{\varepsilon}(x) - f_{0}(x))^{2} + (\sqrt{1 - f_{\varepsilon}(x)^{2}} - \sqrt{1 - f_{0}(x)^{2}})^{2}, \quad x \in [0, \infty), \quad \varepsilon >0.
\end{align*}
Note that $\eta_{\varepsilon}$ converges pointwisely to $0$ for any $x \in [0, \infty)$, thus \cref{lem:pert:ftn:mono,lem:Holder:SOT} conclude that 
\begin{align*}
    \lim_{\varepsilon \downarrow 0} \tr \left[ \m{M}_{0}^{*} \m{M}_{0} \eta_{\varepsilon} (\m{P}^{1/2}) \right] = 0.
\end{align*}
Repeating the argument with the UL decomposition (starting with the bottom-right block) yields the symmetric bound with $\m{G}_{0}$, completing the proof:
\begin{align*}
    \left(
    \begin{array}{c|c}
    \m{I} & \m{R}_{\varepsilon} \\
    \hline
    \m{R}_{\varepsilon} & \m{I}
    \end{array}
    \right) = \left(
    \begin{array}{c|c}
    \m{\Delta}_{\varepsilon} & \m{R}_{\varepsilon} \\
    \hline
    \m{0} & \m{I}
    \end{array}
    \right) \left(
    \begin{array}{c|c}
    \m{\Delta}_{\varepsilon} & \m{0} \\
    \hline
    \m{R}_{\varepsilon} & \m{I}
    \end{array}
    \right).
\end{align*}
\end{enumerate}
\end{proof}

The two bounds in \cref{thm:upper:bdds} offer complementary insights. The upper bound of \eqref{eq:W2:pert:bdd2} converges to $0$ always, while not for \eqref{eq:W2:pert:bdd1}. However, one can easily verify that 
\begin{align*}
    \eta_{\varepsilon}(x) > (1+1/\sqrt{2}) \delta_{\varepsilon}(x) > \delta_{\varepsilon}(x), \quad x, \varepsilon > 0,
\end{align*}
thus $\delta_{\varepsilon}(\m{P}^{1/2}) \prec \eta_{\varepsilon}(\m{P}^{1/2})$. Additionally, \eqref{eq:W2:pert:bdd1} is sharp in the trivial case $\m{A} = \m{B}$.

\begin{conjecture}[Universal Bounds] 
While the asymptotic scaling $\c{W}_{2}^{2}(\pi_{\varepsilon}, \pi_{0}) = \c{O}(\varepsilon)$ is guaranteed in the non-degenerate finite-dimensional setting ($\m{A}, \m{B} \succ \m{0}$ on $\c{H} = \b{R}^{d}$) from \eqref{eq:W2:pert:bdd1}, determining the sharp universal constant governing this rate remains an open problem. Specifically, does there exist a dimension-dependent constant $C_{d} > 0$ such that
\begin{align*} 
    \limsup_{\varepsilon \downarrow 0} \frac{\c{W}_{2}^{2}(\pi_{\varepsilon}, \pi_{0})}{\varepsilon} \le C_{d} \quad \text{for all } \m{A}, \m{B} \succ \m{0}? 
\end{align*}
Analytical results for the identity transport case ($\m{A} = \m{B}$) in \eqref{eq:univ:bdd:same} gives a lower bound of $C_{d} \ge d/2$. However, our numerical experiments with randomly generated full-rank covariances (see \cref{fig:conv:simul}) consistently exhibit a strictly smaller rate, typically concentrating around $\approx 0.36 d$. It remains unclear whether the identity case represents a theoretical worst-case scenario, or if the observed lower rates are simply an artifact of the concentration of measure in random matrix ensembles.
\end{conjecture}

\begin{figure}[h!]
    \centering
    \begin{subfigure}[b]{0.45\textwidth}
     \centering
     \includegraphics[width=\textwidth]{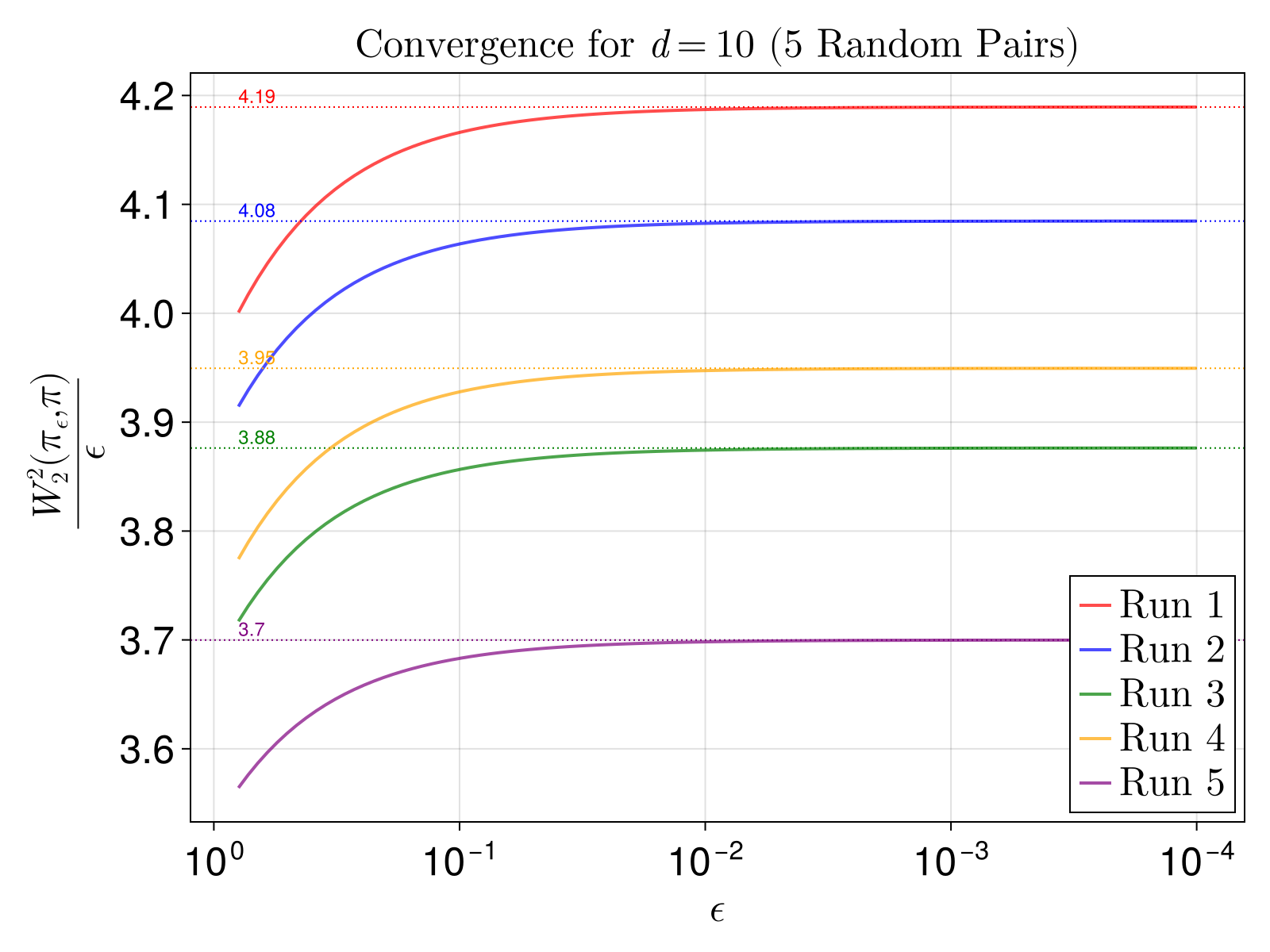}
    \end{subfigure}
    \hfill
    \begin{subfigure}[b]{0.45\textwidth}
     \centering
     \includegraphics[width=\textwidth]{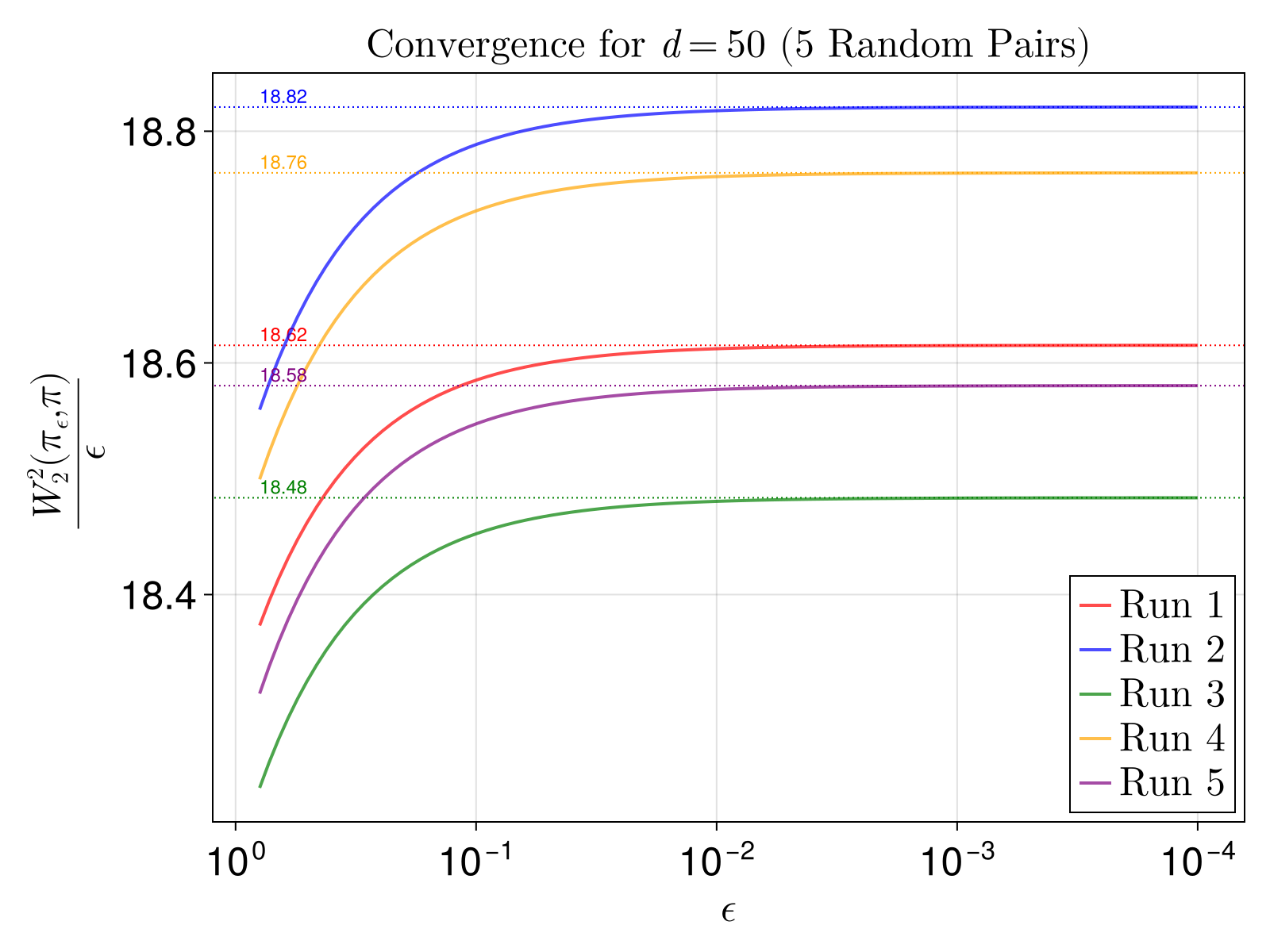}
    \end{subfigure}
    \hfill
    \begin{subfigure}[b]{0.45\textwidth}
     \centering
     \includegraphics[width=\textwidth]{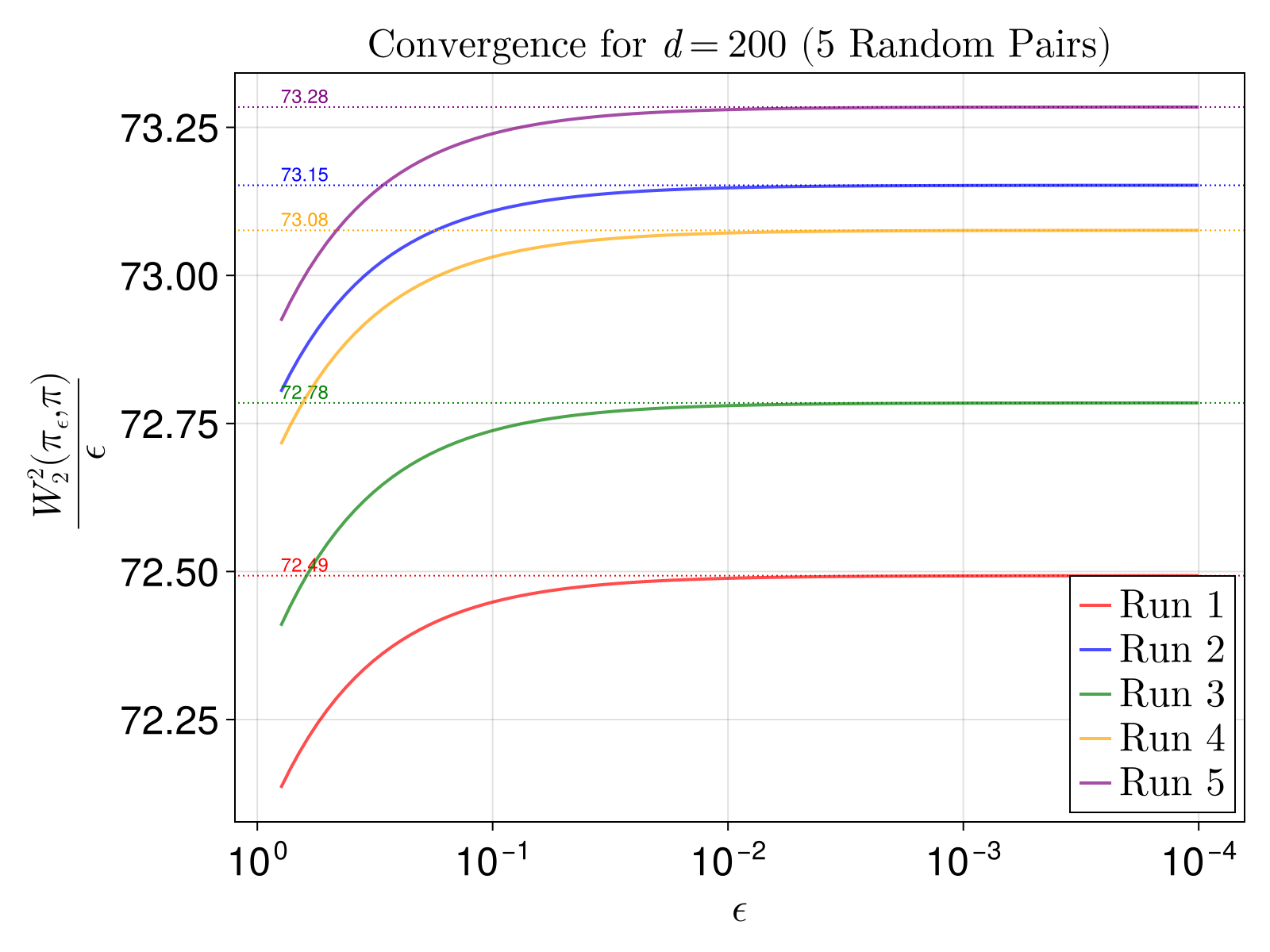}
    \end{subfigure}
    \hfill
    \begin{subfigure}[b]{0.45\textwidth}
     \centering
     \includegraphics[width=\textwidth]{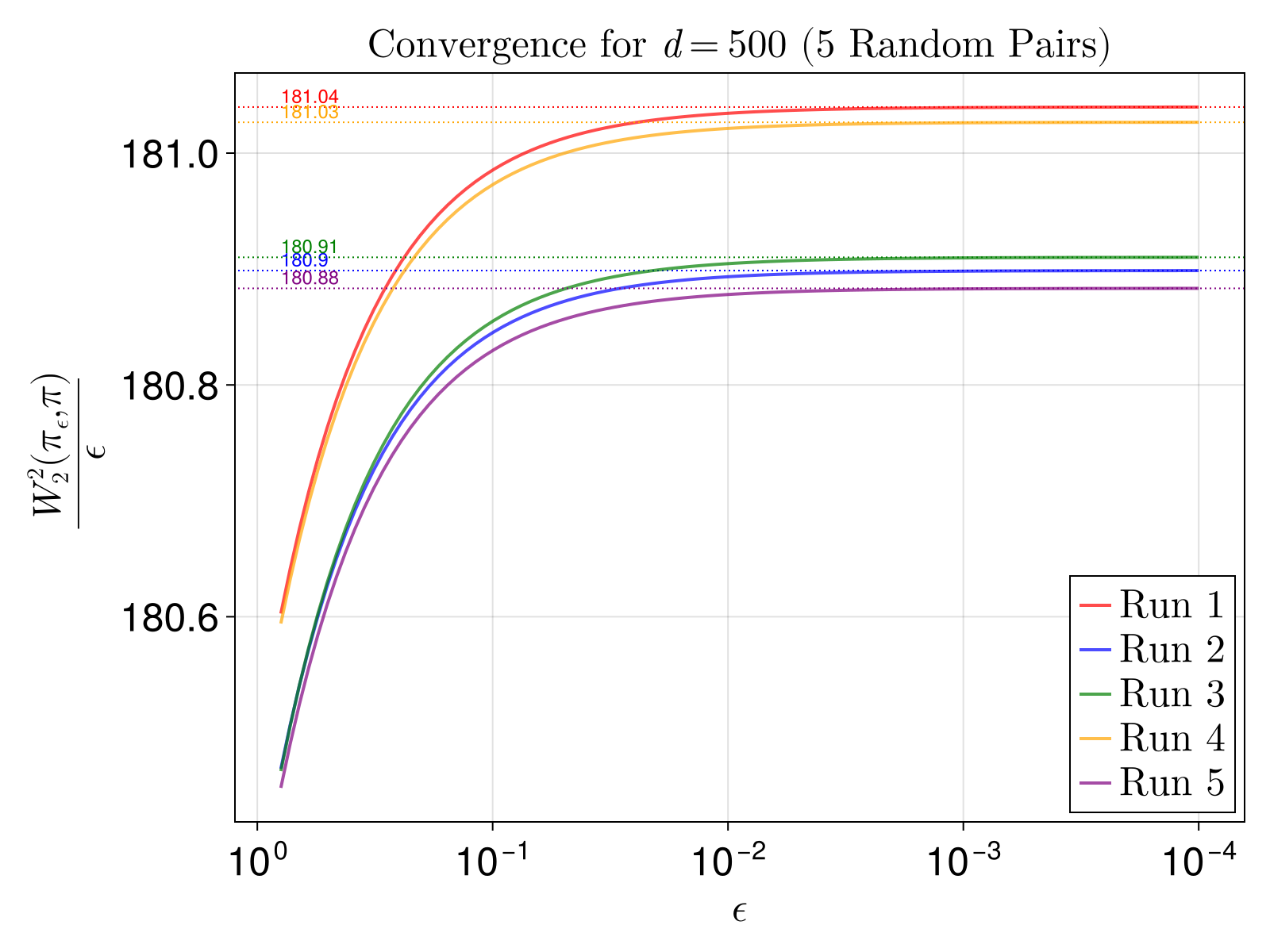}
    \end{subfigure}
    \caption{The ratio $\mathcal{W}_{2}^{2}(\pi_{\varepsilon}, \pi)/\varepsilon$ as a function of the regularization parameter $\varepsilon$ (on a reverse logarithmic scale) for dimensions $d \in \{10, 50, 200, 500\}$. For each dimension, we perform $5$ independent trials with randomly generated full rank covariance matrices $\m{A}, \m{B} \in \b{R}^{d \times d}$. The limiting rate appears to scale linearly with dimension, concentrating around $\approx 0.36d$ for large $d$, which is strictly lower than the theoretical baseline of $d/2$ for identity transport.}
    \label{fig:conv:simul}
\end{figure}

\section{Discussion}
This work establishes an operator-theoretic framework for Gaussian Entropic Optimal Transport, moving beyond the opacity of standard matrix calculus to reveal the problem's intrinsic geometry. By characterizing the solution through properly aligned Green's operators, we demonstrated that the Gaussian EOT problem admits a fully explicit spectral resolution. A notable theoretical consequence of this structure is the gain in regularity inherent to the optimal solution. While the entropic cost remains finite for any Hilbert-Schmidt correlation operator with operator norm strictly less than one, the optimal correlation $\m{R}_{\varepsilon} \in \c{B}_{1}^{+}(\c{H})$ is proven to be a n.n.d. trace-class operator satisfying the explicit bound $\vertii{\m{R}_{\varepsilon}}_1 \le \frac{2}{\varepsilon} \tr[(\m{A}^{1/2} \m{B} \m{A}^{1/2})^{1/2}]$. Our stability analysis also bridges the gap between finite-dimensional and infinite-dimensional theory. While we recover standard parametric rates for finite-rank operators, our derivation of the non-parametric rate for processes like Integrated Brownian Motion quantifies the price of roughness. This result highlights that in the Hilbert space setting, the convergence of the entropic approximation is strictly governed by the spectral decay of the marginals, serving as a cautionary tale against applying finite-dimensional convergence intuitions to functional data.

\begin{acks}[Acknowledgments]
The author is grateful to Victor Panaretos and Paul Freulon for stimulating discussions and valuable insights on the problem.
\end{acks}

\bibliographystyle{imsart-number}
\bibliography{bibliography} 

\begin{thebibliography}{40}
% BibTex style file: imsart-number.bst, 2017-11-03
% Default style options (sort=1,type=number).
% Used options (sort=1,type=number).

\bibitem{baker1970mutual}
\begin{barticle}[author]
\bauthor{\bsnm{Baker},~\bfnm{Charles~R}\binits{C.~R.}}
(\byear{1970}).
\btitle{Mutual information for Gaussian processes}.
\bjournal{SIAM Journal on Applied Mathematics}
\bvolume{19}
\bpages{451--458}.
\end{barticle}
\endbibitem

\bibitem{bhatia2019bures}
\begin{barticle}[author]
\bauthor{\bsnm{Bhatia},~\bfnm{Rajendra}\binits{R.}},
  \bauthor{\bsnm{Jain},~\bfnm{Tanvi}\binits{T.}} \AND
  \bauthor{\bsnm{Lim},~\bfnm{Yongdo}\binits{Y.}}
(\byear{2019}).
\btitle{On the Bures--Wasserstein distance between positive definite matrices}.
\bjournal{Expositiones Mathematicae}
\bvolume{37}
\bpages{165--191}.
\end{barticle}
\endbibitem

\bibitem{bogachev1998gaussian}
\begin{bbook}[author]
\bauthor{\bsnm{Bogachev},~\bfnm{Vladimir~Igorevich}\binits{V.~I.}}
(\byear{1998}).
\btitle{Gaussian measures}
\bvolume{62}.
\bpublisher{American Mathematical Soc.}
\end{bbook}
\endbibitem

\bibitem{bunne2023schrodinger}
\begin{binproceedings}[author]
\bauthor{\bsnm{Bunne},~\bfnm{Charlotte}\binits{C.}},
  \bauthor{\bsnm{Hsieh},~\bfnm{Ya-Ping}\binits{Y.-P.}},
  \bauthor{\bsnm{Cuturi},~\bfnm{Marco}\binits{M.}} \AND
  \bauthor{\bsnm{Krause},~\bfnm{Andreas}\binits{A.}}
(\byear{2023}).
\btitle{The schr{\"o}dinger bridge between gaussian measures has a closed
  form}.
In \bbooktitle{International Conference on Artificial Intelligence and
  Statistics}
\bpages{5802--5833}.
\bpublisher{PMLR}.
\end{binproceedings}
\endbibitem

\bibitem{carlier2017convergence}
\begin{barticle}[author]
\bauthor{\bsnm{Carlier},~\bfnm{Guillaume}\binits{G.}},
  \bauthor{\bsnm{Duval},~\bfnm{Vincent}\binits{V.}},
  \bauthor{\bsnm{Peyr{\'e}},~\bfnm{Gabriel}\binits{G.}} \AND
  \bauthor{\bsnm{Schmitzer},~\bfnm{Bernhard}\binits{B.}}
(\byear{2017}).
\btitle{Convergence of entropic schemes for optimal transport and gradient
  flows}.
\bjournal{SIAM Journal on Mathematical Analysis}
\bvolume{49}
\bpages{1385--1418}.
\end{barticle}
\endbibitem

\bibitem{carlier2023convergence}
\begin{barticle}[author]
\bauthor{\bsnm{Carlier},~\bfnm{Guillaume}\binits{G.}},
  \bauthor{\bsnm{Pegon},~\bfnm{Paul}\binits{P.}} \AND
  \bauthor{\bsnm{Tamanini},~\bfnm{Luca}\binits{L.}}
(\byear{2023}).
\btitle{Convergence rate of general entropic optimal transport costs}.
\bjournal{Calculus of Variations and Partial Differential Equations}
\bvolume{62}
\bpages{116}.
\end{barticle}
\endbibitem

\bibitem{chen2021optimal}
\begin{barticle}[author]
\bauthor{\bsnm{Chen},~\bfnm{Yongxin}\binits{Y.}},
  \bauthor{\bsnm{Georgiou},~\bfnm{Tryphon~T}\binits{T.~T.}} \AND
  \bauthor{\bsnm{Pavon},~\bfnm{Michele}\binits{M.}}
(\byear{2021}).
\btitle{Optimal transport in systems and control}.
\bjournal{Annual Review of Control, Robotics, and Autonomous Systems}
\bvolume{4}
\bpages{89--113}.
\end{barticle}
\endbibitem

\bibitem{chizat2020faster}
\begin{barticle}[author]
\bauthor{\bsnm{Chizat},~\bfnm{Lenaic}\binits{L.}},
  \bauthor{\bsnm{Roussillon},~\bfnm{Pierre}\binits{P.}},
  \bauthor{\bsnm{L{\'e}ger},~\bfnm{Flavien}\binits{F.}},
  \bauthor{\bsnm{Vialard},~\bfnm{Fran{\c{c}}ois-Xavier}\binits{F.-X.}} \AND
  \bauthor{\bsnm{Peyr{\'e}},~\bfnm{Gabriel}\binits{G.}}
(\byear{2020}).
\btitle{Faster Wasserstein distance estimation with the Sinkhorn divergence}.
\bjournal{Advances in neural information processing systems}
\bvolume{33}
\bpages{2257--2269}.
\end{barticle}
\endbibitem

\bibitem{conway2019course}
\begin{bbook}[author]
\bauthor{\bsnm{Conway},~\bfnm{John~B}\binits{J.~B.}}
(\byear{2019}).
\btitle{A course in functional analysis}
\bvolume{96}.
\bpublisher{Springer}.
\end{bbook}
\endbibitem

\bibitem{csiszar1975divergence}
\begin{barticle}[author]
\bauthor{\bsnm{Csiszar},~\bfnm{I.}\binits{I.}}
(\byear{1975}).
\btitle{{$I$-Divergence Geometry of Probability Distributions and Minimization
  Problems}}.
\bjournal{The Annals of Probability}
\bvolume{3}
\bpages{146 -- 158}.
\end{barticle}
\endbibitem

\bibitem{cuturi2013sinkhorn}
\begin{barticle}[author]
\bauthor{\bsnm{Cuturi},~\bfnm{Marco}\binits{M.}}
(\byear{2013}).
\btitle{Sinkhorn distances: Lightspeed computation of optimal transport}.
\bjournal{Advances in neural information processing systems}
\bvolume{26}.
\end{barticle}
\endbibitem

\bibitem{da2006introduction}
\begin{bbook}[author]
\bauthor{\bsnm{Da~Prato},~\bfnm{Giuseppe}\binits{G.}}
(\byear{2006}).
\btitle{An introduction to infinite-dimensional analysis}.
\bpublisher{Springer Science \& Business Media}.
\end{bbook}
\endbibitem

\bibitem{douglas1966majorization}
\begin{barticle}[author]
\bauthor{\bsnm{Douglas},~\bfnm{Ronald~G}\binits{R.~G.}}
(\byear{1966}).
\btitle{On majorization, factorization, and range inclusion of operators on
  Hilbert space}.
\bjournal{Proceedings of the American Mathematical Society}
\bvolume{17}
\bpages{413--415}.
\end{barticle}
\endbibitem

\bibitem{dryden2016statistical}
\begin{bbook}[author]
\bauthor{\bsnm{Dryden},~\bfnm{Ian~L}\binits{I.~L.}} \AND
  \bauthor{\bsnm{Mardia},~\bfnm{Kanti~V}\binits{K.~V.}}
(\byear{2016}).
\btitle{Statistical shape analysis: with applications in R}.
\bpublisher{John Wiley \& Sons}.
\end{bbook}
\endbibitem

\bibitem{eckstein2022quantitative}
\begin{barticle}[author]
\bauthor{\bsnm{Eckstein},~\bfnm{Stephan}\binits{S.}} \AND
  \bauthor{\bsnm{Nutz},~\bfnm{Marcel}\binits{M.}}
(\byear{2022}).
\btitle{Quantitative stability of regularized optimal transport and convergence
  of sinkhorn's algorithm}.
\bjournal{SIAM Journal on Mathematical Analysis}
\bvolume{54}
\bpages{5922--5948}.
\end{barticle}
\endbibitem

\bibitem{freulon2025entropic}
\begin{barticle}[author]
\bauthor{\bsnm{Freulon},~\bfnm{Paul}\binits{P.}},
  \bauthor{\bsnm{Georgakis},~\bfnm{Nikitas}\binits{N.}} \AND
  \bauthor{\bsnm{Panaretos},~\bfnm{Victor}\binits{V.}}
(\byear{2025}).
\btitle{Entropic optimal transport beyond product reference couplings: the
  Gaussian case on Euclidean space}.
\bjournal{arXiv preprint arXiv:2507.01709}.
\end{barticle}
\endbibitem

\bibitem{gelbrich1990formula}
\begin{barticle}[author]
\bauthor{\bsnm{Gelbrich},~\bfnm{Matthias}\binits{M.}}
(\byear{1990}).
\btitle{On a formula for the L2 Wasserstein metric between measures on
  Euclidean and Hilbert spaces}.
\bjournal{Mathematische Nachrichten}
\bvolume{147}
\bpages{185--203}.
\end{barticle}
\endbibitem

\bibitem{genevay2019sample}
\begin{binproceedings}[author]
\bauthor{\bsnm{Genevay},~\bfnm{Aude}\binits{A.}},
  \bauthor{\bsnm{Chizat},~\bfnm{L{\'e}naic}\binits{L.}},
  \bauthor{\bsnm{Bach},~\bfnm{Francis}\binits{F.}},
  \bauthor{\bsnm{Cuturi},~\bfnm{Marco}\binits{M.}} \AND
  \bauthor{\bsnm{Peyr{\'e}},~\bfnm{Gabriel}\binits{G.}}
(\byear{2019}).
\btitle{Sample complexity of sinkhorn divergences}.
In \bbooktitle{The 22nd international conference on artificial intelligence and
  statistics}
\bpages{1574--1583}.
\bpublisher{PMLR}.
\end{binproceedings}
\endbibitem

\bibitem{ghosal2025convergence}
\begin{barticle}[author]
\bauthor{\bsnm{Ghosal},~\bfnm{Promit}\binits{P.}} \AND
  \bauthor{\bsnm{Nutz},~\bfnm{Marcel}\binits{M.}}
(\byear{2025}).
\btitle{On the convergence rate of Sinkhorn’s algorithm}.
\bjournal{Mathematics of Operations Research}.
\end{barticle}
\endbibitem

\bibitem{ghosal2022stability}
\begin{barticle}[author]
\bauthor{\bsnm{Ghosal},~\bfnm{Promit}\binits{P.}},
  \bauthor{\bsnm{Nutz},~\bfnm{Marcel}\binits{M.}} \AND
  \bauthor{\bsnm{Bernton},~\bfnm{Espen}\binits{E.}}
(\byear{2022}).
\btitle{Stability of entropic optimal transport and Schr{\"o}dinger bridges}.
\bjournal{Journal of Functional Analysis}
\bvolume{283}
\bpages{109622}.
\end{barticle}
\endbibitem

\bibitem{gohberg2012traces}
\begin{bbook}[author]
\bauthor{\bsnm{Gohberg},~\bfnm{Israel}\binits{I.}},
  \bauthor{\bsnm{Goldberg},~\bfnm{Seymour}\binits{S.}} \AND
  \bauthor{\bsnm{Krupnik},~\bfnm{Nahum}\binits{N.}}
(\byear{2012}).
\btitle{Traces and determinants of linear operators}
\bvolume{116}.
\bpublisher{Birkh{\"a}user}.
\end{bbook}
\endbibitem

\bibitem{gohberg1978introduction}
\begin{bbook}[author]
\bauthor{\bsnm{Gohberg},~\bfnm{Israel}\binits{I.}} \AND
  \bauthor{\bsnm{Krein},~\bfnm{Mark~Grigorʹevich}\binits{M.~G.}}
(\byear{1978}).
\btitle{Introduction to the theory of linear nonselfadjoint operators}
\bvolume{18}.
\bpublisher{American Mathematical Soc.}
\end{bbook}
\endbibitem

\bibitem{goldfeld2020gaussian}
\begin{binproceedings}[author]
\bauthor{\bsnm{Goldfeld},~\bfnm{Ziv}\binits{Z.}} \AND
  \bauthor{\bsnm{Greenewald},~\bfnm{Kristjan}\binits{K.}}
(\byear{2020}).
\btitle{Gaussian-smoothed optimal transport: Metric structure and statistical
  efficiency}.
In \bbooktitle{International Conference on Artificial Intelligence and
  Statistics}
\bpages{3327--3337}.
\bpublisher{PMLR}.
\end{binproceedings}
\endbibitem

\bibitem{hairer2009introduction}
\begin{barticle}[author]
\bauthor{\bsnm{Hairer},~\bfnm{Martin}\binits{M.}}
(\byear{2009}).
\btitle{An introduction to stochastic PDEs}.
\bjournal{arXiv preprint arXiv:0907.4178}.
\end{barticle}
\endbibitem

\bibitem{hall2013quantum}
\begin{bbook}[author]
\bauthor{\bsnm{Hall},~\bfnm{Brian~C}\binits{B.~C.}}
(\byear{2013}).
\btitle{Quantum theory for mathematicians}
\bvolume{267}.
\bpublisher{Springer Science \& Business Media}.
\end{bbook}
\endbibitem

\bibitem{janati2020entropic}
\begin{barticle}[author]
\bauthor{\bsnm{Janati},~\bfnm{Hicham}\binits{H.}},
  \bauthor{\bsnm{Muzellec},~\bfnm{Boris}\binits{B.}},
  \bauthor{\bsnm{Peyr{\'e}},~\bfnm{Gabriel}\binits{G.}} \AND
  \bauthor{\bsnm{Cuturi},~\bfnm{Marco}\binits{M.}}
(\byear{2020}).
\btitle{Entropic optimal transport between unbalanced gaussian measures has a
  closed form}.
\bjournal{Advances in neural information processing systems}
\bvolume{33}
\bpages{10468--10479}.
\end{barticle}
\endbibitem

\bibitem{leonard2012schrodinger}
\begin{barticle}[author]
\bauthor{\bsnm{L{\'e}onard},~\bfnm{Christian}\binits{C.}}
(\byear{2012}).
\btitle{From the Schr{\"o}dinger problem to the Monge--Kantorovich problem}.
\bjournal{Journal of Functional Analysis}
\bvolume{262}
\bpages{1879--1920}.
\end{barticle}
\endbibitem

\bibitem{mallasto2022entropy}
\begin{barticle}[author]
\bauthor{\bsnm{Mallasto},~\bfnm{Anton}\binits{A.}},
  \bauthor{\bsnm{Gerolin},~\bfnm{Augusto}\binits{A.}} \AND
  \bauthor{\bsnm{Minh},~\bfnm{H{\`a}~Quang}\binits{H.~Q.}}
(\byear{2022}).
\btitle{Entropy-regularized 2-Wasserstein distance between Gaussian measures}.
\bjournal{Information Geometry}
\bvolume{5}
\bpages{289--323}.
\end{barticle}
\endbibitem

\bibitem{masarotto2019procrustes}
\begin{barticle}[author]
\bauthor{\bsnm{Masarotto},~\bfnm{Valentina}\binits{V.}},
  \bauthor{\bsnm{Panaretos},~\bfnm{Victor~M}\binits{V.~M.}} \AND
  \bauthor{\bsnm{Zemel},~\bfnm{Yoav}\binits{Y.}}
(\byear{2019}).
\btitle{Procrustes metrics on covariance operators and optimal transportation
  of Gaussian processes}.
\bjournal{Sankhya A}
\bvolume{81}
\bpages{172--213}.
\end{barticle}
\endbibitem

\bibitem{minh2023entropic}
\begin{barticle}[author]
\bauthor{\bsnm{Minh},~\bfnm{Ha~Quang}\binits{H.~Q.}}
(\byear{2023}).
\btitle{Entropic regularization of Wasserstein distance between
  infinite-dimensional Gaussian measures and Gaussian processes}.
\bjournal{Journal of Theoretical Probability}
\bvolume{36}
\bpages{201--296}.
\end{barticle}
\endbibitem

\bibitem{nutz2022entropic}
\begin{barticle}[author]
\bauthor{\bsnm{Nutz},~\bfnm{Marcel}\binits{M.}} \AND
  \bauthor{\bsnm{Wiesel},~\bfnm{Johannes}\binits{J.}}
(\byear{2022}).
\btitle{Entropic optimal transport: Convergence of potentials}.
\bjournal{Probability Theory and Related Fields}
\bvolume{184}
\bpages{401--424}.
\end{barticle}
\endbibitem

\bibitem{pal2019difference}
\begin{barticle}[author]
\bauthor{\bsnm{Pal},~\bfnm{Soumik}\binits{S.}}
(\byear{2019}).
\btitle{On the difference between entropic cost and the optimal transport
  cost}.
\bjournal{arXiv preprint arXiv:1905.12206}.
\end{barticle}
\endbibitem

\bibitem{panaretos2020invitation}
\begin{bbook}[author]
\bauthor{\bsnm{Panaretos},~\bfnm{Victor~M}\binits{V.~M.}} \AND
  \bauthor{\bsnm{Zemel},~\bfnm{Yoav}\binits{Y.}}
(\byear{2020}).
\btitle{An invitation to statistics in Wasserstein space}.
\bpublisher{Springer Nature}.
\end{bbook}
\endbibitem

\bibitem{paulsen2016introduction}
\begin{bbook}[author]
\bauthor{\bsnm{Paulsen},~\bfnm{Vern~I}\binits{V.~I.}} \AND
  \bauthor{\bsnm{Raghupathi},~\bfnm{Mrinal}\binits{M.}}
(\byear{2016}).
\btitle{An introduction to the theory of reproducing kernel Hilbert spaces}
\bvolume{152}.
\bpublisher{Cambridge university press}.
\end{bbook}
\endbibitem

\bibitem{peyre2019computational}
\begin{barticle}[author]
\bauthor{\bsnm{Peyr{\'e}},~\bfnm{Gabriel}\binits{G.}},
  \bauthor{\bsnm{Cuturi},~\bfnm{Marco}\binits{M.}} \betal{et~al.}
(\byear{2019}).
\btitle{Computational optimal transport: With applications to data science}.
\bjournal{Foundations and Trends{\textregistered} in Machine Learning}
\bvolume{11}
\bpages{355--607}.
\end{barticle}
\endbibitem

\bibitem{pigoli2014distances}
\begin{barticle}[author]
\bauthor{\bsnm{Pigoli},~\bfnm{Davide}\binits{D.}},
  \bauthor{\bsnm{Aston},~\bfnm{John~AD}\binits{J.~A.}},
  \bauthor{\bsnm{Dryden},~\bfnm{Ian~L}\binits{I.~L.}} \AND
  \bauthor{\bsnm{Secchi},~\bfnm{Piercesare}\binits{P.}}
(\byear{2014}).
\btitle{Distances and inference for covariance operators}.
\bjournal{Biometrika}
\bvolume{101}
\bpages{409--422}.
\end{barticle}
\endbibitem

\bibitem{takatsu2011wasserstein}
\begin{barticle}[author]
\bauthor{\bsnm{Takatsu},~\bfnm{Asuka}\binits{A.}}
(\byear{2011}).
\btitle{{Wasserstein geometry of Gaussian measures}}.
\bjournal{Osaka Journal of Mathematics}
\bvolume{48}
\bpages{1005 -- 1026}.
\end{barticle}
\endbibitem

\bibitem{villani2021topics}
\begin{bbook}[author]
\bauthor{\bsnm{Villani},~\bfnm{C{\'e}dric}\binits{C.}}
(\byear{2021}).
\btitle{Topics in optimal transportation}
\bvolume{58}.
\bpublisher{American Mathematical Soc.}
\end{bbook}
\endbibitem

\bibitem{villani2008optimal}
\begin{bbook}[author]
\bauthor{\bsnm{Villani},~\bfnm{C{\'e}dric}\binits{C.}} \betal{et~al.}
(\byear{2008}).
\btitle{Optimal transport: old and new}
\bvolume{338}.
\bpublisher{Springer}.
\end{bbook}
\endbibitem

\bibitem{yun2025gaussian}
\begin{barticle}[author]
\bauthor{\bsnm{Yun},~\bfnm{Ho}\binits{H.}} \AND
  \bauthor{\bsnm{Zemel},~\bfnm{Yoav}\binits{Y.}}
(\byear{2025}).
\btitle{Gaussian Optimal Transport Beyond Brenier's Theorem}.
\bjournal{arXiv preprint}.
\end{barticle}
\endbibitem

\end{thebibliography}

\newpage

\begin{appendix}
\section{Lemmas}

\begin{lemma}\label{lem:Holder:SOT}
Let $\m{G}, \m{M} \in \c{B}_{2}(\c{H})$. If $\{\m{T}_{n}\}_{n}$ is a sequence of bounded operators such that $\m{T}_{n} \to \m{T}$ in the SOT, then the product converges in the trace norm:
\begin{align*}
    \m{G} \m{T}_{n} \m{M}^{*} \to \m{G} \m{T} \m{M}^{*} \quad \text{in} \quad \c{B}_{1}(\c{H}).
\end{align*}
\end{lemma}
\begin{proof}
Let $\m{\Delta}_{n} := \m{T}_{n} - \m{T}$. Since $\m{\Delta}_{n} \xrightarrow{SOT} \m{0}$, the sequence of pointwise norms $\|\m{\Delta}_n x\|$ is bounded for every $x \in \c{H}$. By the Uniform Boundedness Principle, the sequence is uniformly bounded in the operator norm: $C := \sup_n \vertiii{\m{\Delta}_n}_{\infty} < \infty$.
We must show $\vertii{\m{G} \m{\Delta}_{n} \m{M}^{*}}_{1} \to 0$. Since finite-rank operators are dense in $\c{B}_2(\c{H})$ with respect to $\vertii{\cdot}_2$, for any $\delta > 0$, there exist finite-rank operators $\m{G}_F, \m{M}_F$ such that $\vertii{\m{G} - \m{G}_F}_{2} < \delta$ and $\vertii{\m{M} - \m{M}_F}_{2} < \delta$.
We decompose the error term and apply the generalized Hölder inequality $\vertii{\m{A}\m{B}\m{C}}_{1} \le \vertii{\m{A}}_{2} \vertiii{\m{B}}_{\infty} \vertii{\m{C}}_{2}$:
\begin{align*}
    \vertii{\m{G} \m{\Delta}_{n} \m{M}^{*}}_{1} &\le \vertii{(\m{G} - \m{G}_F) \m{\Delta}_{n} \m{M}^{*}}_{1} + \vertii{\m{G}_F \m{\Delta}_{n} (\m{M}^{*} - \m{M}_F^{*})}_{1} + \vertii{\m{G}_F \m{\Delta}_{n} \m{M}_F^{*}}_{1} \\
    &\le \vertii{\m{G} - \m{G}_F}_{2} \vertiii{\m{\Delta}_n}_{\infty} \vertii{\m{M}}_{2} + \vertii{\m{G}_F}_{2} \vertiii{\m{\Delta}_n}_{\infty} \vertii{\m{M} - \m{M}_F}_{2} + \vertii{\m{G}_F \m{\Delta}_{n} \m{M}_F^{*}}_{1} \\
    &\le \delta \cdot C \cdot \vertii{\m{M}}_{2} + \vertii{\m{G}_F}_{2} \cdot C \cdot \delta + \vertii{\m{G}_F \m{\Delta}_{n} \m{M}_F^{*}}_{1}.
\end{align*}
The last term $\vertii{\m{G}_F \m{\Delta}_{n} \m{M}_F^{*}}_{1}$ involves a finite sum of rank-one operators. Specifically, writing the canonical form $\m{G}_F = \sum_{i=1}^k u_i \otimes v_i$ and $\m{M}_F = \sum_{j=1}^l a_j \otimes b_j$, the term is a sum of operators scaled by inner products $\innpr{\m{\Delta}_n b_j}{v_i}$. Since $\m{\Delta}_n \xrightarrow{SOT} \m{0}$, these scalar coefficients vanish as $n \to \infty$. Thus:
\begin{align*}
    \limsup_{n \to \infty} \vertii{\m{G} \m{\Delta}_{n} \m{M}^{*}}_{1} \le \delta C (\vertii{\m{M}}_{2} + \vertii{\m{G}_F}_{2}).
\end{align*}
Since $\delta > 0$ is arbitrary, the limit is zero.
\end{proof}

\begin{lemma}\label{lem:conv:ftn:mono}
Consider the function in \cref{cor:conv:dist}:
\begin{align*}
    g_{1}(x) = \left[ 2 x (1-f_{1}(x)) - \log (1- f_{1}(x)^{2}) \right], \quad f_{1}(x) = \sqrt{1 + \left(\frac{1}{2x}\right)^2} - \frac{1}{2x}, \quad x > 0.
\end{align*}
Then, $g_{1}$ is monotonely increasing and we have
\begin{align*}
    \lim_{x \downarrow 0} \frac{g_{1}(x)}{2x} = 1, \quad \lim_{x \to \infty} (g_{1}(x) - \log x) = 1 - \log 2.
\end{align*}
\end{lemma}
\begin{proof}
\begin{enumerate}[leftmargin = *]
\item It is convenient to express $g_1$ in terms of $f_1$. We use the identity $1 - f_{1}(x)^2 = \frac{f_1(x)}{x}$. 
Let $y = f_1(x)$. Since $x \in (0, \infty)\mapsto f_1(x) \in (0, 1)$ is strictly decreasing, it suffices to show $g_1$ increases with $y$. Substituting $x = \frac{y}{1-y^2}$ into the expression for $g_1$:
\begin{align*}
    \tilde{g_1}(y) = g_{1}(x) = 2 \left( \frac{y}{1-y^2} \right) (1-y) - \log(1-y^2) = \frac{2y}{1+y} - \log(1-y^2).
\end{align*}
Differentiating with respect to $y$:
\begin{align*}
    \frac{\rd \tilde{g_1}}{\rd y} = \frac{2(1+y) - 2y}{(1+y)^2} - \frac{1}{1-y^2}(-2y) = \frac{2}{(1+y)^2} + \frac{2y}{(1-y)(1+y)} > 0,
\end{align*}
i.e., $g_1(x)$ is monotonically increasing.

\item For the limits, first observe that:
\begin{align*}
    \lim_{x \downarrow 0} \frac{g_1(x)}{2x} = \lim_{x \downarrow 0} \left( (1-f_1(x)) - \frac{\log(1-f_1(x)^2)}{2x} \right) = 1 - \lim_{x \downarrow 0} \frac{-f_1(x)^2}{2x} = 1 - 0 = 1.
\end{align*}
For the other limit, let $u = 1/(2x)$. As $x \to \infty$, $u \to 0$. We can rewrite $f_1(x)$ as:
\begin{align*}
    f_{1}(x) = \frac{1/u}{\sqrt{1/u^2+1}+1} = \frac{1}{\sqrt{1+u^2}+u} = \sqrt{1+u^2}-u.
\end{align*}
For small $u$, we have the expansion $f_1 \approx (1 + u^2/2) - u = 1 - u + O(u^2)$.
This implies $1 - f_1 \approx u$.
Now substitute into $g_1(x)$:
\begin{align*}
    g_{1}(x) &= \frac{1}{u}(1-f_1) - \log(1-f_1^2).
\end{align*}
Using $1-f_1^2 = (1-f_1)(1+f_1) \approx u(2-u) \approx 2u$, we get:
\begin{align*}
    g_{1}(x) &\approx \frac{1}{u}(u) - \log(2u) = 1 - (\log 2 + \log u) = 1 - \log 2 - \log(1/(2x)) = 1 + \log x.
\end{align*}
Thus, $\lim_{x \to \infty} (g_1(x) - \log x) = 1$.
\end{enumerate}
\end{proof}

\begin{lemma}\label{lem:pert:ftn:mono}
Consider the perturbation function in \eqref{eq:def:pert:ftn}. Then, the following holds:
\begin{enumerate}[leftmargin = *]
\item For any $\varepsilon > 0$, $\delta_{\varepsilon} : [0, \infty) \to [0, 2-\sqrt{2})$ is monotonely decreasing in $x$ with $\lim_{x \downarrow 0} \delta_{\varepsilon}(x) = 2 -\sqrt{2}$. For fixed $x > 0$, $\delta_{\varepsilon}(x) \downarrow 0$ as $\varepsilon \downarrow 0$.
\item For any $\varepsilon, x > 0$, we have the scaling property:
\begin{align*}
    \frac{x \delta_{\varepsilon}(x)}{\varepsilon} = \psi \left( \frac{\varepsilon}{2x} \right), 
\end{align*}
where
\begin{align*}
    \psi:(0, \infty) \to \left[0, \frac{1}{4} \right], \quad \psi(u) := 
        \frac{1}{2u} \left( 2 - \sqrt{2\left(1 + \sqrt{1+u^2}-u \right)} \right) , \quad u > 0,
\end{align*}
is a monotonely decreasing function. The limiting behavior of $\psi$ is given by
\begin{align*}
    \lim_{u \downarrow 0} \psi(u) = \frac{1}{4}, \quad \lim_{u \uparrow \infty} \psi(u) = 0.
\end{align*}
\end{enumerate}
\end{lemma}
\begin{proof}
\begin{enumerate}[leftmargin = *]
\item This is trivial since the spectral shrinkage function
\begin{align*}
    f_{\varepsilon}(x) = \frac{2x}{\sqrt{4x^{2}+\varepsilon^{2}}+\varepsilon}= \frac{1}{\sqrt{1 + \left(\frac{\varepsilon}{2x}\right)^2} + \frac{\varepsilon}{2x}}, \quad x>0
\end{align*}
is monotonely increasing in $x > 0$ and monotonely decreasing in $\varepsilon > 0$.

\item Let $u := \frac{\varepsilon}{2x}$. We rewrite the scaled function solely in terms of $u$:
\begin{align*}
    \frac{x \delta_{\varepsilon}(x)}{\varepsilon} = \frac{1}{2u} \left( 2 - \sqrt{2\left(1 + \sqrt{1+u^2}-u \right)} \right) =: \psi(u).
\end{align*}
For the continuity of $\psi(u)$, observe that
\begin{align*}
    \sqrt{2\left(1 + \sqrt{1+u^2}-u \right)} = \sqrt{4 - 2u + \c{O}(u^{2})} = 2 - \frac{u}{2} + \c{O}(u^{2}), 
\end{align*}
thus
\begin{align*}
    \lim_{u \downarrow 0} \psi(u) = \lim_{u \downarrow 0} \frac{\frac{u}{2} + \c{O}(u^{2})}{2u} = \frac{1}{4}.
\end{align*}
Then, $\lim_{u \uparrow \infty} \psi(u) = 0$ is trivial.
To establish the monotonicity, we use the change of variable $v = \sqrt{1+u^2}-u \in (0, 1)$, which is monotnonely decreasing in $u$. Substituting $v$ into $\psi(u)$:
\begin{align*} 
    \psi(u) = \frac{v}{1-v^2} \left( 2 - \sqrt{2(1+v)} \right). 
\end{align*}
By applying the change of variable $y = \sqrt{1+v} \in (1, \sqrt{2})$ once again, it remains to show that
\begin{align*}
    \tilde{\psi}(y) = \psi(u) = \frac{(y^{2}-1)}{y^2(2-y^2)} \left( 2 - \sqrt{2} y \right) = \frac{\sqrt{2} (y^{2}-1)}{y^2(\sqrt{2}+y)}  
\end{align*}
is monotonely increasing in $y$. By taking the logarithmic derivative with respect to $y$:
\begin{align*}
    \frac{\rd}{\rd y} \log \tilde{\psi}(y) = \frac{2y}{y^2-1} - \frac{2}{y} - \frac{1}{\sqrt{2}+y} = \frac{2}{y (y^2-1)} - \frac{1}{\sqrt{2}+y}
\end{align*}
the condition becomes: 
\begin{align*} 
    2(\sqrt{2}+y) > y (y^2-1) \iff h(y):= y^3 - 3y - 2\sqrt{2} < 0 \quad \text{on} \quad (1, \sqrt{2}). 
\end{align*}
This condition holds since $h(y)$ is increasing in $y \in (1, \sqrt{2})$ and $h(\sqrt{2}) = -3 \sqrt{2} < 0$.
\end{enumerate}
\end{proof}

\end{appendix}
\end{document}